\documentclass[11pt]{amsart}
\usepackage{amsmath,amssymb,mathtools}
\usepackage{microtype}
\usepackage[a4paper,margin=29mm,headheight=14pt]{geometry}
\usepackage{xcolor}
\usepackage[colorlinks=true,linkcolor=blue!45!black,citecolor=blue!45!black,urlcolor=blue!45!black]{hyperref}
\numberwithin{equation}{section}
\newtheorem{theorem}{Theorem}[section]
\newtheorem{lemma}[theorem]{Lemma}
\newtheorem{proposition}[theorem]{Proposition}
\newtheorem{corollary}[theorem]{Corollary}
\theoremstyle{definition}
\newtheorem{definition}[theorem]{Definition}
\theoremstyle{remark}
\newtheorem{remark}[theorem]{Remark}
\allowdisplaybreaks[1]
\date{}
\title[Coefficient duality for flag Besov spaces]{Coefficient duality and interpolation of Euclidean flag Besov spaces}
\author[C. Li]{Chang Li}
\address{Department of Mathematics, Shanghai University,
Shanghai 200444, People's Republic of China}
\email{changli@shu.edu.cn}
\begin{document}
\begin{abstract}
The main purpose of this paper is to develop the duality theory of
the homogeneous Besov spaces associated with the Euclidean flag,
using the discrete Littlewood--Paley--Stein analysis. We first
establish an exact discrete sampling formula together with bounded
analysis and synthesis maps, by which the space is identified with a
sequence space of coefficients. The continuous dual is then computed
through a coefficient norm. In this computation the geometry of the
sampling rectangles plays an essential role: when the spatial
exponent is below one, the volume of a dyadic rectangle forces an
explicit shift of both smoothness indices. We devote particular
attention to the supremum sequence exponent. At that endpoint the
closure of the flag test functions is characterized by vanishing
scale tails, and the limiting tail is comparable to the distance to
this closure; every continuous functional on the full space
decomposes into a canonical regular part and a part that annihilates
the closure, and the latter may be nonzero. For finite exponents the
dual is again a flag Besov space. As applications, finite
Calder\'on sums give a quantitative approximation theorem, and the
common analysis and synthesis framework yields real and complex
interpolation as well as the lifting property. The arguments are
based on the exact sampling formula, scalar maximal estimates and
the transpose of one bounded coefficient projection.
\end{abstract}
\maketitle

\section{Introduction and statement of results}
\label{sec:intro}

The duality theory of Besov spaces is naturally expressed in
terms of the order in which the norm measures position and scale.
In a flag decomposition, the spatial coefficients at each pair of
scales form an $\ell^p$ block, and the sizes of these blocks form
an $\ell^q$ sequence. The main purpose of this paper is to use this
block structure to describe all continuous functionals on a
homogeneous flag Besov space, including the quasi-Banach range in
which the spatial exponent is below one and the endpoint at which
the sequence exponent is infinite. The geometry of the sampling
rectangles enters explicitly in this description: the volume of a
dyadic rectangle determines the shift of the smoothness indices in
the dual space.

The flag considered here is
\[
 \{0\}\subset\{0\}\times\mathbb R^m
                  \subset\mathbb R^n\times\mathbb R^m.
\]
Its Littlewood--Paley pieces combine an isotropic decomposition
on $\mathbb R^{n+m}$ with a second decomposition in the $y$
variables. This geometry arises naturally in the multiparameter
multiplier theory of M\"uller, Ricci and Stein
\cite{MRS95,MRS96} and in the theory of flag kernels developed by
Nagel, Ricci and Stein \cite{NRS01}; see also \cite{NRSW12} for
homogeneous groups. The product theories of Chang and Fefferman
\cite{CF80} and Journ\'e \cite{Journe85} provide the corresponding
multiparameter precedents.

The coefficient method has its classical foundations in
Triebel \cite{Triebel83}, Frazier and Jawerth \cite{FJ85,FJ90},
and in the reproducing-formula methods of Han and Sawyer
\cite{HS94} and Han \cite{Han00}. In the flag setting, Han and
Lu \cite{HL08,HL10} established the discrete
Littlewood--Paley--Stein theory and the associated Hardy spaces,
and Ding, Lu and Ma \cite{DLM10} developed the multiparameter
Besov and Triebel--Lizorkin spaces associated with flag singular
integrals. In the present paper we work in this Euclidean flag
geometry, specify the smoothness weights explicitly, and prove
the sampling estimates and the analysis--synthesis bounds needed
in the sequel. In particular, the Plancherel--P\'olya comparison
and the independence of the analysing pair follow directly from
these sampling estimates. Further developments of the flag Hardy
theory may be found in \cite{HLLW22,HLW19}, while inhomogeneous
characterisations were studied in \cite{WL12}. More recently,
Dou \cite{Dou26} proved the boundedness of classical
Calder\'on--Zygmund convolution operators on Euclidean flag
Besov spaces.

Yang \cite{Yang09} developed duality, embeddings and lifting in a
parabolic flag geometry, in which the first-layer dilation scales
the two groups of variables with different powers; we refer in
particular to Proposition~2.11 for Besov duality in the finite
Banach range, to Remark~2.12 for test-function closures at
supremum endpoints, and to Theorem~4.6 for a lifting result. In
the present paper the first-layer dilation is isotropic on
$\mathbb R^{n+m}$. The volume of a dyadic rectangle then
determines the change of smoothness when $p<1$, and the endpoint
argument gives a quantitative description of the little space
together with a regular--singular decomposition of the full dual.

We would like to point out that at the supremum sequence exponent
one must distinguish density from bounded approximation. For the
spaces defined below, test functions are dense when
$0<p,q<\infty$. When $0<p<\infty$ and $q=\infty$, the closure
consists precisely of those elements whose scale norms vanish at
infinity (Theorem~\ref{thm:little}), and the example in the proof
of Theorem~\ref{thm:decomposition} shows that this closure can be
proper. We estimate the distance from an element to the closure by
its limiting scale tail, identify the dual of the closure, and
separate the canonical coefficient functional from the part of a
full-space functional that annihilates the closure.

The same block structure also applies to interpolation. At fixed
spatial exponent and smoothness, the real $K$-functional reduces
to a scalar sequence $K$-functional, which yields real
interpolation in the sequence exponent throughout the stated
quasi-Banach range. Complex interpolation in the finite Banach
range and the lifting property follow from the common analysis and
synthesis maps together with the localised kernel bounds. Here and
below we use the interpolation framework of
\cite{BL76,BS88,Holmstedt70}.

\subsection{Spaces and dyadic frames}

A flag rectangle is a product $R=I\times J$, where
$I\subset\mathbb R^n$ and $J\subset\mathbb R^m$ are cubes with
sides parallel to the axes and
$\ell(I)\leq\ell(J)$. The flag maximal function is
\begin{equation}\label{eq:MF}
 M_F h(z)=\sup_{\substack{R\ni z\\R\text{ a flag rectangle}}}
                        \frac1{|R|}\int_R|h(u)|\,du.
\end{equation}

Let $\mathcal S_\infty(\mathbb R^{n+m}\times\mathbb R^m)$ consist
of the Schwartz functions $F(x,y,z)$ satisfying
\begin{equation}\label{eq:Sinfty}
 \int_{\mathbb R^{n+m}}F(x,y,z)x^\alpha y^\beta\,dx\,dy=0,
 \qquad
 \int_{\mathbb R^m}F(x,y,z)z^\gamma\,dz=0
\end{equation}
for all multi-indices and all values of the remaining variables.
Following Han--Lu \cite{HL08}, set
\begin{equation}\label{eq:pi}
 \pi F(x,y)=\int_{\mathbb R^m}F(x,y-z,z)\,dz,\qquad
 \mathcal S_F=\pi\bigl(\mathcal S_\infty
              (\mathbb R^{n+m}\times\mathbb R^m)\bigr),
\end{equation}
with the quotient Schwartz topology, and denote its continuous
dual by $\mathcal S_F'$. Pairings are linear in the distribution
and conjugate linear in the test function.

The partial-cancellation formulation is used in
\cite[Definition~1.2]{HLW19}. The following proposition identifies it,
including its topology, with the projection definition above.

\begin{proposition}[Test-space identification]\label{def:test}
The flag test space $\mathcal S_F$ is the closed subspace of
$\mathcal S(\mathbb R^{n+m})$ characterized by
\begin{equation}\label{eq:partialmoments}
 \int_{\mathbb R^m}f(x,y)y^\beta\,dy=0
 \quad\text{for every }x\in\mathbb R^n
 \text{ and every multi-index }\beta.
\end{equation}
Its quotient Schwartz topology agrees with the topology inherited
from $\mathcal S(\mathbb R^{n+m})$.
\end{proposition}
\begin{proof}
Let $\mathcal Z_y$ denote the subspace in
\eqref{eq:partialmoments}. It is closed, since the partial-moment
maps are continuous on Schwartz space. Equivalently,
$\widehat f(\xi,\eta)$ and all its $\eta$ derivatives vanish on
$\eta=0$. The projection $\pi$ is continuous on Schwartz spaces,
and
\[
 \widehat{\pi F}(\xi,\eta)=\widehat F(\xi,\eta,\eta).
\]
The product moment conditions imply flatness of $\widehat F$ on
$\zeta=0$, so differentiation of this restriction gives
$\pi(\mathcal S_\infty)\subset\mathcal Z_y$.

For the reverse inclusion choose
$\rho\in C_c^\infty((1/4,4))$ with $\rho(1)=1$. For
$f\in\mathcal Z_y$ define $\mathcal E f$ on the Fourier side by
\[
 \widehat{\mathcal E f}(\xi,\eta,\zeta)=
 \begin{cases}
  \widehat f(\xi,\eta)\rho(|\zeta|^2/|\eta|^2),&\eta\ne0,\\
  0,&\eta=0.
 \end{cases}
\]
On its support $|\eta|/2\leq|\zeta|\leq2|\eta|$, and a derivative
of order $d$ of the cutoff is bounded by $C_d|\eta|^{-d}$.
For any multi-indices $\alpha,\beta$ and integers $K,M\geq0$,
Taylor's formula and the flatness of $\widehat f$ give
\[
 |\partial_\xi^\alpha\partial_\eta^\beta\widehat f(\xi,\eta)|
 \leq C\,\mathfrak p(f)|\eta|^K(1+|\xi|)^{-M},
 \qquad 0<|\eta|\leq1,
\]
where $\mathfrak p$ is a continuous Schwartz seminorm depending on
the displayed indices. Taking $K$ arbitrarily large absorbs every
cutoff derivative and gives a smooth extension, flat on $\eta=0$.
For $|\eta|\geq1$, the support relation and the ordinary Schwartz
estimates control all derivatives and polynomial weights in
$(\xi,\eta,\zeta)$. Thus $\mathcal E$ is a continuous linear map
from $\mathcal Z_y$ into Schwartz space. Its Fourier transform is
flat on both $\zeta=0$ and $(\xi,\eta)=0$, so
$\mathcal E f\in\mathcal S_\infty$. The restriction identity and
$\rho(1)=1$ give $\pi\mathcal E f=f$.

The induced map from
$\mathcal S_\infty/\ker\pi$ to $\mathcal Z_y$ is continuous and
bijective, and its inverse is the continuous map sending $f$ to
the class of $\mathcal E f$. This proves the topology assertion.
\end{proof}

By Hahn--Banach every continuous functional on this closed
Schwartz subspace extends to a tempered distribution. Consequently
$\mathcal S_F'$ is canonically the quotient of $\mathcal S'$ by
the annihilator of $\mathcal S_F$. Convolution with a translate of
an annular flag analysing function is independent of the chosen
tempered representative, since all such translates belong to
$\mathcal S_F$.

Choose real even Schwartz functions $\psi^{(1)}$ on
$\mathbb R^{n+m}$ and $\psi^{(2)}$ on $\mathbb R^m$, with Fourier
transforms supported in the annulus $\{1/2\leq|\cdot|\leq2\}$ in
their respective spaces; their moments of every order then
vanish. Suppose that
\begin{align}
 \sum_{j\in\mathbb Z}
  |\widehat{\psi^{(1)}}(2^{-j}\xi,2^{-j}\eta)|^2&=1
       &&((\xi,\eta)\ne0),\label{eq:partition1}\\
 \sum_{k\in\mathbb Z}
  |\widehat{\psi^{(2)}}(2^{-k}\eta)|^2&=1
       &&(\eta\ne0).\label{eq:partition2}
\end{align}
We call such a pair admissible. To construct one, take a
nonnegative smooth radial cutoff supported in the annulus,
positive on a smaller annulus whose dyadic dilates cover the
punctured space, and divide it by the square root of the sum of
the squares of its dyadic dilates. Put
\begin{equation}\label{eq:psijk}
 \psi_{j,k}(x,y)
 =\int_{\mathbb R^m}2^{j(n+m)}\psi^{(1)}
        \bigl(2^jx,2^j(y-z)\bigr)\,
                  2^{km}\psi^{(2)}(2^kz)\,dz .
\end{equation}
The notation $\varphi_{j,k}$ has the same meaning for another
admissible pair. A bump function $\psi_{j,k}$ vanishes unless
$k\leq j+1$.

For $s=(s_1,s_2)\in\mathbb R^2$, we use the smoothness weight
$2^{-js_1}2^{-(j+k)s_2}=2^{-j(s_1+s_2)}2^{-ks_2}$.
On the support of a nonzero bump function,
$|(\xi,\eta)|\asymp2^j$ and $|\eta|\asymp2^k$; thus the
exponents attached to these two frequency scales are
$s_1+s_2$ and $s_2$, respectively, with the displayed negative-sign
convention. We shall use this parametrisation throughout the paper.

\begin{definition}\label{def:B}
Let $s\in\mathbb R^2$, $0<p<\infty$ and $0<q\leq\infty$. The
homogeneous flag Besov space
$\dot B^{s,q}_{p,F}=\dot B^{s,q}_{p,F}(\mathbb R^n\times\mathbb R^m)$
consists of those $f\in\mathcal S_F'$ for which
\begin{equation}\label{eq:Bdefinition}
 \|f\|_{\dot B^{s,q}_{p,F}}
 =\left\{\sum_{j,k\in\mathbb Z}
 \bigg(2^{-js_1}2^{-(j+k)s_2}\|\psi_{j,k}*f\|_{L^p}\bigg)^q
                  \right\}^{1/q}<\infty,
\end{equation}
with the outer expression a supremum when $q=\infty$.
\end{definition}

The effective lengths and the sampling grids are
\begin{equation}\label{eq:grid}
 a_j=2^{-j},\qquad b_{j,k}=2^{-j}\vee2^{-k},\qquad
 \mathcal R_{j,k}^N
 =\{I\times J:\ell(I)=2^{-N}a_j,\
                    \ell(J)=2^{-N}b_{j,k}\},
\end{equation}
where $I,J$ range over the usual half-open dyadic cubes and
$z_R$ denotes the lower corner of $R$. We always fix an integer
$N\geq N_0\geq0$ sufficiently large for the Fourier support boxes
and the sampling absorptions in Section~\ref{sec:discrete}.
The threshold may depend on the admissible pairs and the
spatial exponents under consideration, but not on $j,k$ or
the functions being sampled. It is chosen also for the
uniformly enlarged support boxes used in regularisation.
Whenever several spaces occur in one statement, and in particular
for an interpolation couple and its asserted intermediate space,
we choose one common $N$ satisfying all these requirements.
At a fixed scale each
rectangle has volume
\begin{equation}\label{eq:Rvolume}
 |R|=2^{-N(n+m)}2^{-jn}2^{-m\min(j,k)}
 \asymp_N 2^{-jn}2^{-km}
\end{equation}
for a nonzero bump function, uniformly in $j,k$.

\begin{definition}\label{def:b}
For a labelled sequence $t=\{t_{j,k,R}\}$ write
$t_{j,k}=\sum_R t_{j,k,R}\mathbf 1_R$. The coefficient space
$\dot b^{s,q}_{p,F}$ is defined by
\begin{equation}\label{eq:bsequence}
 \begin{aligned}
 \|t\|_{\dot b^{s,q}_{p,F}}
 &=\left\{\sum_{j,k}
       \bigg(2^{-js_1}2^{-(j+k)s_2}\|t_{j,k}\|_{L^p}\bigg)^q\right\}^{1/q}\\
 &=\left\{\sum_{j,k}\left[
       2^{-js_1}2^{-(j+k)s_2}
       \left(\sum_{R\in\mathcal R_{j,k}^N}
               |t_{j,k,R}|^p |R|\right)^{1/p}
                      \right]^q\right\}^{1/q},
 \end{aligned}
\end{equation}
with the outer expression a supremum when $q=\infty$.
\end{definition}

\begin{definition}\label{def:ST}
The analysis and synthesis maps associated with an admissible
pair are
\begin{equation}\label{eq:analysis-synthesis}
 S_\psi f=\{(\psi_{j,k}*f)(z_R)\}_{j,k,R},\qquad
 T_\psi t=\sum_{j,k}\sum_{R\in\mathcal R_{j,k}^N}
       |R|\,t_{j,k,R}\psi_{j,k}(\,\cdot-z_R).
\end{equation}
\end{definition}

\begin{theorem}[Retraction and discrete norm]\label{thm:retraction}
Let $s\in\mathbb R^2$, $0<p<\infty$ and $0<q\leq\infty$. Then
\[
 S_\psi:\dot B^{s,q}_{p,F}\longrightarrow\dot b^{s,q}_{p,F},
 \qquad
 T_\psi:\dot b^{s,q}_{p,F}\longrightarrow\dot B^{s,q}_{p,F}
\]
are bounded, $T_\psi S_\psi=I$, and
\begin{equation}\label{eq:discrete-norm}
 \|f\|_{\dot B^{s,q}_{p,F}}
 \asymp\|S_\psi f\|_{\dot b^{s,q}_{p,F}}
 \asymp\inf_{T_\psi t=f}\|t\|_{\dot b^{s,q}_{p,F}}.
\end{equation}
The space $\dot B^{s,q}_{p,F}$ is quasi-Banach, continuously
embedded in $\mathcal S_F'$, and is Banach when $p,q\geq1$. The
test space $\mathcal S_F$ is continuously contained in it and is
dense when $q<\infty$.
\end{theorem}

\subsection{The little space and the supremum endpoint}

Write $J=(j,k)$ and $|J|_1=|j|+|k|$. For
$0<q\leq\infty$ define the scale tail
\begin{equation}\label{eq:omega}
 \omega_{f,q}(K)
 =\left(\sum_{|J|_1>K}
      [2^{-js_1}2^{-(j+k)s_2}\|\psi_{j,k}*f\|_{L^p}]^q\right)^{1/q},
\end{equation}
with the sum replaced by a supremum when $q=\infty$, and put
$d(f)=\lim_{K\to\infty}\omega_{f,\infty}(K)$.

\begin{theorem}[The little Besov space]\label{thm:little}
Let $s\in\mathbb R^2$ and $0<p<\infty$. For
$f\in\dot B^{s,\infty}_{p,F}$,
\begin{equation}\label{eq:distance}
 d(f)\leq
 \inf_{g\in\mathcal S_F}\|f-g\|_{\dot B^{s,\infty}_{p,F}}
 \leq C d(f).
\end{equation}
Consequently the closure
$B_0=\overline{\mathcal S_F}^{\,\dot B^{s,\infty}_{p,F}}$
consists precisely of those $f$ satisfying
\begin{equation}\label{eq:little}
 \lim_{|J|_1\to\infty}
          2^{-js_1}2^{-(j+k)s_2}\|\psi_{j,k}*f\|_{L^p}=0.
\end{equation}
This condition is independent of the admissible pair.
\end{theorem}

For an exponent $r\in(0,\infty]$ put
\begin{equation}\label{eq:star}
 r^*=\begin{cases}
       \infty,&0<r\leq1,\\
       r/(r-1),&1<r<\infty,\\
       1,&r=\infty.
      \end{cases}
\end{equation}
The value $\infty$ for $r\leq1$ describes the dual of a sequence
space.

\begin{definition}\label{def:associate}
For a coefficient sequence $d=\{d_{J,R}\}$ define the block
dual size
\begin{equation}\label{eq:block-dual}
 D_J(d)=
 \begin{cases}
 \displaystyle\left(\sum_{R\in\mathcal R_J^N}
       |R||d_{J,R}|^{p'}\right)^{1/p'},
                       &1<p<\infty,\quad p'=p/(p-1),\\[4pt]
 \displaystyle\sup_{R\in\mathcal R_J^N}
                  |R|^{1-1/p}|d_{J,R}|,&0<p\leq1.
 \end{cases}
\end{equation}
The associate coefficient norm is
\begin{equation}\label{eq:exact-dual}
 \|d\|_{(\dot b^{s,q}_{p,F})^\times}
       =\bigl\|\{2^{js_1}2^{(j+k)s_2}D_J(d)\}_J\bigr\|_{\ell^{q^*}},
\end{equation}
and the coefficient pairing is
\begin{equation}\label{eq:sequence-pairing}
 [t,d]_\psi=\sum_J\sum_{R\in\mathcal R_J^N}
                         |R|t_{J,R}\overline{d_{J,R}}.
\end{equation}
The corresponding distribution space is
\begin{equation}\label{eq:distribution-dual}
 (\dot B^{s,q}_{p,F})^\times
   =\{g\in\mathcal S_F':S_\psi g\in
                            (\dot b^{s,q}_{p,F})^\times\},
 \qquad
 \|g\|_{(\dot B^{s,q}_{p,F})^\times}
      =\|S_\psi g\|_{(\dot b^{s,q}_{p,F})^\times}.
\end{equation}
\end{definition}

\begin{theorem}[Continuous duals]\label{thm:duality-all}
If $0<p,q<\infty$, then
\begin{equation}\label{eq:full-finite-dual}
 (\dot B^{s,q}_{p,F})^*
                   =(\dot B^{s,q}_{p,F})^\times
\end{equation}
under the pairing
\begin{equation}\label{eq:duality-pairing}
 \langle f,g\rangle_F=[S_\psi f,S_\psi g]_\psi,
\end{equation}
with equivalent norms and unique representatives in
$\mathcal S_F'$. At $q=\infty$,
\begin{equation}\label{eq:little-dual}
 B_0^*=(\dot B^{s,\infty}_{p,F})^\times,
\end{equation}
with the outer sequence norm in the dual of the little space
equal to $\ell^1$.
\end{theorem}

\begin{theorem}[Regular and singular parts]\label{thm:decomposition}
Write $B=\dot B^{s,\infty}_{p,F}$. For $g\in B^\times$ let
$\ell_g(f)=\langle f,g\rangle_F$ be its canonical functional on
the full space. Every $\ell\in B^*$ has a unique decomposition
\begin{equation}\label{eq:regular-singular}
 \ell=\ell_g+\ell_0,\qquad
 g\in B^\times,\quad
 \ell_0\in B_0^\perp:=\{L\in B^*:L|_{B_0}=0\},
\end{equation}
and
\begin{equation}\label{eq:decomp-norm}
 \|\ell\|\asymp\|g\|_{B^\times}+\|\ell_0\|.
\end{equation}
The decomposition and its regular part are independent of the
admissible pair and the fixed grid. In general $B_0^\perp$ is
nonzero.
\end{theorem}

\begin{definition}\label{def:Binfty}
For $t\in\mathbb R^2$ and $0<a\leq\infty$ the spatial supremum
flag Besov space is defined by
\begin{equation}\label{eq:spatial-supremum-Besov}
 \|g\|_{\dot B^{t,a}_{\infty,F}}
   =\left\|\left\{2^{-jt_1}2^{-(j+k)t_2}\|\psi_J*g\|_{L^\infty}\right\}_J
                                                    \right\|_{\ell^a}.
\end{equation}
\end{definition}

\begin{corollary}\label{cor:duality}
If $1<p<\infty$ and $0<q<\infty$, then
\begin{equation}\label{eq:reciprocal-dual}
          (\dot B^{s,q}_{p,F})^*=\dot B^{-s,q^*}_{p',F},
\end{equation}
with equivalent norms and the pairing
\eqref{eq:duality-pairing}; in particular, for $1<p,q<\infty$
the dual is $\dot B^{-s,q'}_{p',F}$. At the supremum endpoint,
\begin{equation}\label{eq:little-dual-p}
 B_0^*=\dot B^{-s,1}_{p',F}.
\end{equation}
\end{corollary}

\begin{corollary}\label{cor:small-p-duality}
Let $0<p\leq1$, put $\delta=1/p-1$ and
\begin{equation}\label{eq:small-p-dual-smoothness}
       t=(-s_1+(m-n)\delta,\,-s_2-m\delta).
\end{equation}
For $0<q<\infty$,
\begin{equation}\label{eq:small-p-duality}
       (\dot B^{s,q}_{p,F})^*=\dot B^{t,q^*}_{\infty,F},
\end{equation}
and $B_0^*=\dot B^{t,1}_{\infty,F}$. For $p=1$ one has $t=-s$.
\end{corollary}

\begin{theorem}[Real interpolation]\label{thm:real-interpolation}
Fix $s\in\mathbb R^2$ and $0<p<\infty$. Let
$0<q_0<q_1\leq\infty$, $0<\theta<1$, and
\begin{equation}\label{eq:real-q}
             \frac1q=\frac{1-\theta}{q_0}+\frac\theta{q_1}.
\end{equation}
Then
\begin{equation}\label{eq:real-interpolation}
 \bigl(\dot B^{s,q_0}_{p,F},
       \dot B^{s,q_1}_{p,F}\bigr)_{\theta,q}
             =\dot B^{s,q}_{p,F}
\end{equation}
with equivalent quasi-norms. The fine index is the displayed $q$;
the smoothness and spatial exponent are fixed.
\end{theorem}

\begin{theorem}[Complex interpolation]\label{thm:complex-interpolation}
Let $1<p_0,p_1,q_0,q_1<\infty$, $s^0,s^1\in\mathbb R^2$ and
$0<\theta<1$. Put
\begin{equation}\label{eq:interpolated-parameters}
 \frac1p=\frac{1-\theta}{p_0}+\frac\theta{p_1},\qquad
 \frac1q=\frac{1-\theta}{q_0}+\frac\theta{q_1},\qquad
 s=(1-\theta)s^0+\theta s^1.
\end{equation}
Then
\begin{equation}\label{eq:complex-interpolation}
 \bigg(\dot B^{s^0,q_0}_{p_0,F},
       \dot B^{s^1,q_1}_{p_1,F}\bigg)_\theta
                       =\dot B^{s,q}_{p,F}
\end{equation}
with equivalent norms.
\end{theorem}

For $\alpha=(\alpha_1,\alpha_2)\in\mathbb R^2$ define
\begin{equation}\label{eq:lifting-symbol}
 m_\alpha(\xi,\eta)=|(\xi,\eta)|^{\alpha_1+\alpha_2}
                              |\eta|^{\alpha_2},
\end{equation}
and let $\mathcal D^\alpha$ be the Fourier multiplier
$\widehat{\mathcal D^\alpha f}=m_\alpha\widehat f$.

\begin{theorem}[Lifting]\label{thm:lifting}
For $s,\alpha\in\mathbb R^2$, $0<p<\infty$ and
$0<q\leq\infty$, $\mathcal D^\alpha$ is an isomorphism
\begin{equation}\label{eq:lifting}
 \mathcal D^\alpha:\dot B^{s,q}_{p,F}
                   \longrightarrow\dot B^{s+\alpha,q}_{p,F},
 \qquad
 \|\mathcal D^\alpha f\|_{\dot B^{s+\alpha,q}_{p,F}}
                           \asymp\|f\|_{\dot B^{s,q}_{p,F}},
\end{equation}
with inverse $\mathcal D^{-\alpha}$; it restricts to an
isomorphism of the corresponding little spaces.
\end{theorem}

The paper is organized as follows. Section~\ref{sec:discrete}
proves Theorem~\ref{thm:retraction}. Section~\ref{sec:tails}
treats the scale tails and proves Theorem~\ref{thm:little}.
Section~\ref{sec:duality} proves Theorems~\ref{thm:duality-all}
and~\ref{thm:decomposition} and Corollaries~\ref{cor:duality}
and~\ref{cor:small-p-duality}. Section~\ref{sec:interpolation}
proves Theorems~\ref{thm:real-interpolation}
and~\ref{thm:complex-interpolation}. Section~\ref{sec:lifting}
proves Theorem~\ref{thm:lifting}. Appendix~\ref{app:mixed}
records the mixed-norm comparisons and the resulting
Hardy-space embeddings.

\section{Discrete analysis and synthesis: proof of Theorem~\ref{thm:retraction}}
\label{sec:discrete}

Throughout the paper $D=n+m$, and $C$ denotes a positive
constant which may change from line to line; its dependence on
the indices of the space and finitely many seminorms of the
functions under consideration is allowed. We use the Fourier
transform
$\widehat f(\xi,\eta)=\int e^{-i(x\cdot\xi+y\cdot\eta)}f(x,y)\,dx\,dy$.

The continuous Calder\'on reproducing formula
\begin{equation}\label{eq:continuousCRF}
 f=\sum_{j,k}\psi_{j,k}*\psi_{j,k}*f
\end{equation}
holds on $\mathcal S_F$ and on $\mathcal S_F'$. Indeed, for $\eta\ne0$,
\begin{align}\label{eq:partition-product-expanded}
 \sum_{j,k}|\widehat\psi_{j,k}(\xi,\eta)|^2
 &=\left(\sum_j|\widehat{\psi^{(1)}}(2^{-j}\xi,2^{-j}\eta)|^2\right)
   \left(\sum_k|\widehat{\psi^{(2)}}(2^{-k}\eta)|^2\right)=1.
\end{align}
The flatness of test functions at $\eta=0$ permits passage to
the test-space limit; transposition gives the distributional
identity. This is the continuous formula of \cite{HL08,DLM10}.

We record the exact sampling argument, following
\cite[Theorem 2.1]{HLW19}. For an annular pair, the Fourier
support of both $\psi_{j,k}$ and $\psi_{j,k}*f$ is contained in
$|\xi|\leq 2a_j^{-1}$, $|\eta|\leq 2b_{j,k}^{-1}$. The product,
as a function of $u$, of $\psi_{j,k}(z-u)$ and
$(\psi_{j,k}*f)(u)$ therefore has Fourier support
$|\xi|\leq 4a_j^{-1}$, $|\eta|\leq 4b_{j,k}^{-1}$, which lies
strictly inside the first reciprocal lattice point
$2\pi\,2^N a_j^{-1}$, $2\pi\,2^N b_{j,k}^{-1}$ for the chosen
$N\geq N_0$. To make the absence of aliasing explicit, write
$\delta=2^{-N}$ and, initially for a test function $f$, put
\[
 F_z(u)=\psi_{j,k}(z-u)(\psi_{j,k}*f)(u),\qquad
 V_{j,k}=(\delta a_j)^n(\delta b_{j,k})^m.
\]
Poisson summation with our Fourier convention reads
\begin{align}\label{eq:poisson-expanded}
 &V_{j,k}\sum_{\ell\in\mathbb Z^n}\sum_{h\in\mathbb Z^m}
       F_z(\delta a_j\ell,\delta b_{j,k}h)\notag\\
 &\quad=\sum_{\ell\in\mathbb Z^n}\sum_{h\in\mathbb Z^m}
       \widehat F_z\left(\frac{2\pi\ell}{\delta a_j},
                         \frac{2\pi h}{\delta b_{j,k}}\right)
       =\widehat F_z(0,0)=\int F_z(u)\,du.
\end{align}
Every term on the second line except $(\ell,h)=(0,0)$ vanishes
by the support restriction. Since $V_{j,k}=|R|$, this gives
\begin{equation}\label{eq:sampling-identity}
 \int\psi_{j,k}(z-u)(\psi_{j,k}*f)(u)\,du
 =\sum_{R\in\mathcal R_{j,k}^N}
   |R|\psi_{j,k}(z-z_R)(\psi_{j,k}*f)(z_R).
\end{equation}
Summing in $j,k$ proves the exact reproducing formula
\begin{equation}\label{eq:exactCRF}
 f=\sum_{j,k}\sum_{R\in\mathcal R_{j,k}^N}
       |R|\,\psi_{j,k}(\,\cdot-z_R)\,
                    (\psi_{j,k}*f)(z_R).
\end{equation}
For $\phi\in\mathcal S_F$, annular localisation and flatness
give, for arbitrary $L,M>0$,
\begin{equation}\label{eq:testcoeff}
 |(\psi_{j,k}*\phi)(z)|
 \leq C_{L,M,\phi}\,2^{-L(|j|+|k|)}(1+|z|)^{-M};
\end{equation}
one differentiates $\widehat\psi_{j,k}\widehat\phi$ a prescribed
number of times, using Taylor's formula at low $\eta$ frequencies
and Schwartz decay at high total frequencies. Together with the
polynomial growth of the translated kernels this proves the
convergence of \eqref{eq:exactCRF} in $\mathcal S_F$ for test
functions and weakly in $\mathcal S_F'$ for distributions.

The sampling identity and Plancherel give, for $f\in L^2$,
\begin{equation}\label{eq:L2isometry}
 \sum_{j,k}\sum_{R\in\mathcal R_{j,k}^N}
      |R|\,|(\psi_{j,k}*f)(z_R)|^2=\|f\|_2^2.
\end{equation}

We shall prove the following Plancherel--P\'olya comparison
for two admissible pairs $\psi,\varphi$ and a sufficiently
fine common grid:
\begin{equation}\label{eq:PP}
 \begin{aligned}
 &\left\|\left\{2^{-js_1}2^{-(j+k)s_2}
   \sum_R\sup_R|\psi_{j,k}*f|\,\mathbf 1_R\right\}\right\|_{\ell^q(L^p)}\\
 &\qquad\asymp
 \left\|\left\{2^{-js_1}2^{-(j+k)s_2}
   \sum_R\inf_R|\varphi_{j,k}*f|\,\mathbf 1_R\right\}\right\|_{\ell^q(L^p)}.
 \end{aligned}
\end{equation}
The annular supports give the kernel estimate
\begin{equation}\label{eq:almost}
 |\psi_{j,k}*\varphi_{j',k'}(z)|
 \leq C_{L,M}2^{-L(|j-j'|+|k-k'|)}P_{A,B,M}(z),
\end{equation}
where $A=\max(a_j,a_{j'})$, $B=\max(b_{j,k},b_{j',k'})$ and
\begin{equation}\label{eq:P}
 P_{A,B,M}(x,y)=
 \frac{A^M}{(A+|x|)^{n+M}}
 \frac{B^M}{(B+|y|)^{m+M}}.
\end{equation}
Indeed, the Fourier transform of the convolution vanishes
unless $|j-j'|+|k-k'|$ is bounded by a fixed constant, since
both its joint-frequency annuli and its $y$-frequency annuli
must intersect. More precisely,
\begin{equation}\label{eq:finite-overlap-expanded}
 \widehat\psi_{j,k}\widehat\varphi_{j',k'}\ne0
 \quad\Longrightarrow\quad
 2^{-2}\leq2^{j-j'}\leq2^2,\qquad
 2^{-2}\leq2^{k-k'}\leq2^2.
\end{equation}
For the remaining indices, rescale $x$ by
$A$ and $y$ by $B$. The resulting Fourier multipliers have
uniformly bounded derivatives of every order and support in
a fixed box: in particular, the ratio $a_j/b_{j,k}$ is at
most one. Integration by parts gives the product decay
\eqref{eq:P}. Since only bounded differences of scales
remain, the factor $2^{-L(|j-j'|+|k-k'|)}$ can be inserted
for any $L>0$ by changing the constant. This proves
\eqref{eq:almost} without restrictions on smoothness.

We next establish the estimates needed for the synthesis map.

\begin{lemma}\label{lem:maximal}
For $1<u<\infty$,
\begin{equation}\label{eq:scalar-maximal}
                    \|M_Fh\|_{L^u}\leq C_u\|h\|_{L^u}.
\end{equation}
\end{lemma}
\begin{proof}
Write $M_{n+m}$ for the ordinary maximal operator on
$\mathbb R^{n+m}$ and $M_m$ for the maximal operator in the
$y$ variable. One has
\begin{equation}\label{eq:flag-domination}
                         M_Fh\leq C M_m(M_{n+m}h).
\end{equation}
Indeed, let $R=I\times J$ contain $(x,y)$, with side lengths
$a\leq b$, and let $Q(v,a)$ be the $y$ cube centred at $v$ with
side length $a$. A cube of side at most $2a$ containing
$I\times Q(v,a)$ also contains $(x,v)$, so
\[
 M_{n+m}h(x,v)\geq c a^{-(n+m)}
                 \int_I\int_{Q(v,a)}|h(u,t)|\,dt\,du.
\]
For every $t\in J$, the set of $v\in3J$ with $t\in Q(v,a)$
has measure $a^m$. Integrating in $v\in3J$ and using Fubini
therefore bounds the average on $R$ by
$C|3J|^{-1}\int_{3J}M_{n+m}h(x,v)\,dv$.
Since $y\in3J$, this proves \eqref{eq:flag-domination}.
The scalar maximal theorem and Fubini give the norm estimate
\begin{align}\label{eq:maximal-norm-chain}
 \|M_Fh\|_{L^u}
 &\leq C\|M_m(M_{n+m}h)\|_{L^u}
 \leq C_u\|M_{n+m}h\|_{L^u}
 \leq C_u\|h\|_{L^u},
\end{align}
which proves \eqref{eq:scalar-maximal}.
\end{proof}

\begin{lemma}\label{lem:grid}
Let $0<r\leq1$ and $M>D/r$. Choose arbitrary points
$z_{R'}\in R'\in\mathcal R_{j',k'}^N$ and nonnegative numbers
$c_{R'}$, and set $H=\sum_{R'}c_{R'}\mathbf1_{R'}$. If
$z,\bar z\in R\in\mathcal R_{j,k}^N$, then
\begin{align}\label{eq:grid-bound}
 \sum_{R'\in\mathcal R_{j',k'}^N} |R'|
          P_{A,B,M}(\bar z-z_{R'})c_{R'}
 \leq C\delta^{-D(1/r-1)}
       \left(\frac A{a'}\right)^{n(1/r-1)}
       \left(\frac B{b'}\right)^{m(1/r-1)}
                     [M_F(H^r)(z)]^{1/r},
\end{align}
where $\delta=2^{-N}$, $a'=a_{j'}$ and $b'=b_{j',k'}$. The two
scale-ratio factors are at most
$2^{D(1/r-1)(|j-j'|+|k-k'|)}$.
\end{lemma}
\begin{proof}
It is enough to consider finitely many nonzero $c_{R'}$. Split
the points into product shells of radii $2^h A$ and $2^l B$,
$h,l\geq0$, with the zeroth shells including distances at most
$2A$, $2B$. On the $(h,l)$ shell the kernel is at most
$C A^{-n}B^{-m}2^{-h(n+M)}2^{-l(m+M)}$, and all source rectangles
in this shell, together with $z$, lie in a flag rectangle
$E_{h,l}$ with side lengths $C2^{\max(h,l)}A$ and
$C2^{\max(h,l)}B$. Since $A\leq B$, this common dilation preserves
the flag condition, and $|E_{h,l}|\leq C\,2^{D\max(h,l)}A^nB^m$.
At the source scale every rectangle has
volume $V=\delta^D(a')^n(b')^m$. Since $r\leq1$,
\begin{equation}\label{eq:shell-step}
 \sum_{R'\text{ in the shell}}Vc_{R'}
 \leq V^{1-1/r}
                  \left(\int_{E_{h,l}}H^r\right)^{1/r}.
\end{equation}
Combining these estimates, the shell contributes at most
\begin{align*}
 C\delta^{-D(1/r-1)}
 \left(\frac A{a'}\right)^{n(1/r-1)}
 \left(\frac B{b'}\right)^{m(1/r-1)}
 2^{-h(n+M)}2^{-l(m+M)}
 2^{D\max(h,l)/r}[M_F(H^r)(z)]^{1/r}.
\end{align*}
The required summability is explicit:
\begin{align}\label{eq:geometric-shell-sum}
 &\sum_{h,l\geq0}2^{-h(n+M)}2^{-l(m+M)}2^{D\max(h,l)/r}\notag\\
 &\quad\leq\left(\sum_{h\geq0}2^{-h(n+M-D/r)}\right)
             \left(\sum_{l\geq0}2^{-l(m+M-D/r)}\right)\notag\\
 &\quad=\frac1{(1-2^{-(n+M-D/r)})(1-2^{-(m+M-D/r)})}<\infty.
\end{align} The last assertion follows from
$\log_2(A/a')=(j'-j)_+$ and
$\log_2(B/b')=(\min(j',k')-\min(j,k))_+$, each bounded by
$|j-j'|+|k-k'|$.
The argument also permits different fixed source and target
grids: the factor $\delta^{-D(1/r-1)}$ then uses the source
mesh, and the constant may depend on both fixed meshes.
\end{proof}

\begin{lemma}\label{lem:matrix}
Suppose nonnegative functions $u_{j,k}$ satisfy, for
$z\in R\in\mathcal R_{j,k}^N$,
\begin{align}\label{eq:matrix-hypothesis}
 u_{j,k}(z)\leq C\sup_{\bar z\in R}\sum_{j',k'}
   2^{-L(|j-j'|+|k-k'|)}
   \sum_{R'\in\mathcal R_{j',k'}^N}|R'|
          P_{A,B,M}(\bar z-z_{R'})c_{j',k',R'}.
\end{align}
Set $H_{j,k}=\sum_R c_{j,k,R}\mathbf1_R$. Let $0<p<\infty$,
$0<q\leq\infty$, choose $0<r<\min(1,p)$, and suppose
$M>D/r$ and $L>|s_1+s_2|+|s_2|+D(1/r-1)$. Then
\begin{equation}\label{eq:matrix-B}
 \|\{2^{-js_1}2^{-(j+k)s_2} u_{j,k}\}\|_{\ell^q(L^p)}
 \leq C\|\{2^{-js_1}2^{-(j+k)s_2} H_{j,k}\}\|_{\ell^q(L^p)}.
\end{equation}
\end{lemma}
\begin{proof}
The weight ratio satisfies
\begin{equation}\label{eq:smoothness-transfer}
 2^{-js_1}2^{-(j+k)s_2}2^{j's_1}2^{(j'+k')s_2}
 \leq 2^{|s_1+s_2||j-j'|+|s_2||k-k'|}.
\end{equation}
Lemma~\ref{lem:grid} therefore gives
\begin{equation}\label{eq:pointwise-matrix}
 2^{-js_1}2^{-(j+k)s_2} u_{j,k}
 \leq C\sum_{j',k'}2^{-\varepsilon(|j-j'|+|k-k'|)}
       [M_F((2^{-j's_1}2^{-(j'+k')s_2}H_{j',k'})^r)]^{1/r},
\end{equation}
with $\varepsilon=L-|s_1+s_2|-|s_2|-D(1/r-1)>0$. Put
$\tau=\min(1,p,q)$ when $q<\infty$ and $\tau=\min(1,p)$ when
$q=\infty$. The $\tau$-triangle inequality in $L^p$ and
Lemma~\ref{lem:maximal} give
\begin{equation}\label{eq:matrix-Lp}
 \|2^{-js_1}2^{-(j+k)s_2} u_{j,k}\|_{L^p}^\tau
 \leq C\sum_{j',k'}2^{-\varepsilon\tau(|j-j'|+|k-k'|)}
                    \|2^{-j's_1}2^{-(j'+k')s_2} H_{j',k'}\|_{L^p}^\tau.
\end{equation}
The scalar maximal step used in the preceding estimate is precisely
\begin{equation}\label{eq:powered-maximal-norm}
 \|[M_F(|H|^r)]^{1/r}\|_{L^p}
 =\|M_F(|H|^r)\|_{L^{p/r}}^{1/r}
 \leq C\||H|^r\|_{L^{p/r}}^{1/r}=C\|H\|_{L^p},
 \qquad p/r>1.
\end{equation}
Put $\kappa_{h,l}=2^{-\varepsilon\tau(|h|+|l|)}$ and let
$A_J$ and $B_J$ denote the two weighted $L^p$ block norms in
\eqref{eq:matrix-Lp}. Then
\begin{align}\label{eq:young-expanded}
 \|\{A_J\}\|_{\ell^q}^{\tau}
 &=\|\{A_J^{\tau}\}\|_{\ell^{q/\tau}}
 \leq C\|\kappa*\{B_J^{\tau}\}\|_{\ell^{q/\tau}}\notag\\
 &\leq C\|\kappa\|_{\ell^1}\|\{B_J^{\tau}\}\|_{\ell^{q/\tau}}
 =C\left(\frac{1+2^{-\varepsilon\tau}}
                  {1-2^{-\varepsilon\tau}}\right)^2
       \|\{B_J\}\|_{\ell^q}^{\tau}.
\end{align}
Here $q/\tau\geq1$; for $q=\infty$ the same calculation uses
$\ell^\infty$. Taking the $\tau$th root proves \eqref{eq:matrix-B}.
\end{proof}

\begin{lemma}\label{lem:peetre}
Let $0<a\leq b$, $0<r\leq1$ and $\lambda r>D$.
If the Fourier support of a
tempered distribution $h$ is contained in
$\{(\xi,\eta):|\xi|\leq C_0/a,\ |\eta|\leq C_0/b\}$, then
\begin{equation}\label{eq:regularized-peetre}
 h^*_{a,b,\lambda}(z)
       \leq C[M_F(|h|^r)(z)]^{1/r},
\end{equation}
where
$\rho_{a,b}(x,y)=\max(|x|/a,|y|/b)$ and
$h^*_{a,b,\lambda}(z)=\sup_u|h(u)|(1+\rho_{a,b}(z-u))^{-\lambda}$.
For $u,z$ in the same rectangle of side lengths $\delta a,\delta b$,
$0<\delta\leq1$,
\begin{equation}\label{eq:normalised-derivative}
 a|\nabla_x h(u)|+b|\nabla_y h(u)|
       \leq C[M_F(|h|^r)(z)]^{1/r}.
\end{equation}
\end{lemma}
\begin{proof}
First suppose that $h$ is a Schwartz function. Choose a
Schwartz reproducing kernel
$\gamma_{a,b}(x,y)=a^{-n}b^{-m}\gamma(x/a,y/b)$ whose Fourier
transform is one on the box with constant $C_0+1$, so that
$h=\gamma_{a,b}*h$. Write
$P=h^*_{a,b,\lambda}(z)<\infty$. The case $h=0$ is immediate.
Using
$|h(v)|^{1-r}\leq P^{1-r}(1+\rho_{a,b}(z-v))^{\lambda(1-r)}$ and
the triangle inequality for $\rho_{a,b}$,
\begin{equation}\label{eq:peetre-step}
 \frac{|h(u)|}{(1+\rho_{a,b}(z-u))^\lambda}
 \leq C P^{1-r}(1+\rho_{a,b}(z-u))^{-\lambda r}
 \int a^{-n}b^{-m}(1+\rho_{a,b}(u-v))^{-T}|h(v)|^r\,dv,
\end{equation}
with $T$ arbitrarily large. Because $a\leq b$, the rectangles
with side lengths $2^l a,2^l b$ are flag rectangles.
Write $d=\rho_{a,b}(z-u)$. Decomposition into these rectangles
centred at $u$, followed by enlargement to contain $z$, gives
\begin{align}\label{eq:peetre-shell-expanded}
 &\int a^{-n}b^{-m}(1+\rho_{a,b}(u-v))^{-T}|h(v)|^r\,dv\notag\\
 &\quad\leq C\sum_{l\geq0}2^{-lT}a^{-n}b^{-m}
             \int_{\rho_{a,b}(u-v)\leq2^{l+1}}|h(v)|^r\,dv\notag\\
 &\quad\leq C\sum_{l\geq0}2^{-lT}(2^l+d)^D M_F(|h|^r)(z)\notag\\
 &\quad\leq C(1+d)^D\left(\sum_{l\geq0}2^{-l(T-D)}\right)
                         M_F(|h|^r)(z).
\end{align}
Choose $T>D$. Since $\lambda r>D$, \eqref{eq:peetre-step} yields
\[
 P\leq C P^{1-r}\sup_{d\geq0}(1+d)^{D-\lambda r}
                  M_F(|h|^r)(z)
 \leq C P^{1-r}M_F(|h|^r)(z).
\]
Dividing by $P^{1-r}$ proves \eqref{eq:regularized-peetre}. Differentiating the reproducing
identity and using the rapid decay of the differentiated
kernel gives
\[
 a|\nabla_xh(u)|+b|\nabla_yh(u)|
 \leq C P(1+\rho_{a,b}(z-u))^\lambda.
\]
For $u,z$ in the same rectangle this proves
\eqref{eq:normalised-derivative}.

For a general band-limited tempered distribution, $h$ is a
smooth function whose derivatives have polynomial growth.
Choose $\vartheta\in\mathcal S$ with
$\widehat\vartheta\in C_c^\infty$ supported in
$\{|\xi|\leq1,\ |\eta|\leq1\}$, $\vartheta(0)=1$, and
$|\vartheta|\leq1$. Such a function is obtained by normalising
the inverse Fourier transform of a nonnegative even cutoff.
For $0<\epsilon\leq1$, set
\[
 h_\epsilon(x,y)
   =h(x,y)\vartheta(\epsilon x/a,\epsilon y/b).
\]
Then $h_\epsilon\in\mathcal S$, its Fourier support lies in
the box with constant $C_0+1$, and
$|h_\epsilon|\leq|h|$. The preceding proof, with a reproducing
kernel equal to one on this enlarged box, has constants
independent of $\epsilon$. Since $h_\epsilon\to h$ locally
with all derivatives, passage to the limit proves the
pointwise derivative bound, and taking the supremum over $u$
proves the Peetre bound. If the maximal function on the right
is infinite, the assertions hold in the extended sense.
\end{proof}

\begin{lemma}\label{lem:Linfty-sampling}
For a sufficiently fine fixed grid,
\begin{equation}\label{eq:Linfty-sampling}
 \|\psi_J*g\|_{L^\infty}
       \asymp\sup_{R\in\mathcal R_J^N}|(\psi_J*g)(z_R)|,
\end{equation}
uniformly in $J$ and $g\in\mathcal S_F'$.
\end{lemma}
\begin{proof}
The sampled supremum is bounded by the continuous one. First
suppose that $h=\psi_J*g$ is bounded. The derivative estimate
\eqref{eq:normalised-derivative} with $r=1$ gives
$\sup_z\{a_j|\nabla_xh(z)|+b_{j,k}|\nabla_yh(z)|\}\leq C\|h\|_\infty$.
The mean value theorem on a rectangle yields
$\|h\|_\infty\leq\sup_R|h(z_R)|+C2^{-N}\|h\|_\infty$; choosing
$N$ so that $C2^{-N}<1/2$ and absorbing proves the reverse
inequality. For a general $h$, use the band-limited
regularisation $h_\epsilon$ from Lemma~\ref{lem:peetre}.
These Schwartz functions have one uniformly enlarged support
box, for which $N_0$ has also been chosen. The same absorption
therefore gives
\[
 \|h_\epsilon\|_\infty
 \leq 2\sup_R|h_\epsilon(z_R)|
 \leq 2\sup_R|h(z_R)|.
\]
Letting $\epsilon\to0$ at each point and then taking the
supremum proves the assertion, also when boundedness of $h$
is not known in advance.
\end{proof}

We are now in a position to prove Theorem~\ref{thm:retraction}.
Choose $0<r<\min(1,p)$. For $z\in R$, Lemma~\ref{lem:peetre}
applied to $h=\psi_{j,k}*f$ gives
\begin{align}\label{eq:analysis-block-expanded}
 \left(\sum_R|h(z_R)|^p|R|\right)^{1/p}
 &=\left\|\sum_R|h(z_R)|\mathbf1_R\right\|_{L^p}\notag\\
 &\leq C\|[M_F(|h|^r)]^{1/r}\|_{L^p}
 \leq C\|h\|_{L^p}.
\end{align}
Multiplying by the smoothness weight and applying the outer
$\ell^q$ quasi-norm, we obtain
\begin{equation}\label{eq:analysis-norm-expanded}
 \|S_\psi f\|_{\dot b^{s,q}_{p,F}}
 \leq C\left\|\left\{2^{-js_1}2^{-(j+k)s_2}
                          \|\psi_{j,k}*f\|_{L^p}\right\}_{j,k}\right\|_{\ell^q}
 =C\|f\|_{\dot B^{s,q}_{p,F}}.
\end{equation}
For the synthesis estimate, a coefficient sequence satisfies
\begin{equation}\label{eq:coefficient-growth}
 |t_{j,k,R}|
 \leq 2^{js_1}2^{(j+k)s_2}|R|^{-1/p}
                         \|t\|_{\dot b^{s,q}_{p,F}},
\end{equation}
and $|R|^{-1/p}\leq C2^{C_1(|j|+|k|)}$. For $M>D$ the lattice
counting bound
\begin{equation}\label{eq:lattice-counting}
 \sum_{R\in\mathcal R_{j,k}^N}|R|(1+|z_R|)^{-M}
 \leq C_M(1+2^{-N}a_j)^n(1+2^{-N}b_{j,k})^m
\end{equation}
holds. In view of \eqref{eq:testcoeff}, the series
\begin{equation}\label{eq:synthesis-abs}
 \sum_{j,k,R}|R|\,|t_{j,k,R}|\,
                  |\langle\psi_{j,k}(\,\cdot-z_R),h\rangle|
\end{equation}
is bounded, for a fixed constant $C_2$ and every sufficiently large
$L$, by
\begin{align}\label{eq:synthesis-summability-expanded}
 &\sum_{j,k,R}|R|\,|t_{j,k,R}|\,
        |\langle\psi_{j,k}(\,\cdot-z_R),h\rangle|\notag\\
 &\quad\leq C\mathfrak p(h)\|t\|_{\dot b^{s,q}_{p,F}}
                    \sum_{j,k}2^{-(L-C_2)(|j|+|k|)}\notag\\
 &\quad\leq C_L\mathfrak p(h)\|t\|_{\dot b^{s,q}_{p,F}},
 \qquad L>C_2.
\end{align}
Here $\mathfrak p$ is a continuous Schwartz seminorm, and $C_2$
absorbs the volume, smoothness and lattice-growth factors.
Thus $T_\psi t$ is well defined and continuous in $\mathcal S_F'$. For finite
$t$, applying any admissible $\varphi_{j,k}$ to the synthesis
formula and using \eqref{eq:almost} and Lemma~\ref{lem:matrix}
gives
\begin{equation}\label{eq:synthesis-bound}
 \|T_\psi t\|_{\dot B^{s,q}_{p,F}}
                  \leq C\|t\|_{\dot b^{s,q}_{p,F}},
\end{equation}
where the norm on the left may be computed with $\varphi$.
The general case follows by finite truncation and Fatou's lemma.
The exact formula \eqref{eq:exactCRF} gives $T_\psi S_\psi=I$,
and the two bounded maps give \eqref{eq:discrete-norm}.
Computing the norm with another admissible pair $\varphi$ gives
\begin{equation}\label{eq:pair-independence-expanded}
 \|f\|_{\dot B^{s,q}_{p,F}(\varphi)}
 =\|T_\psi S_\psi f\|_{\dot B^{s,q}_{p,F}(\varphi)}
 \leq C\|S_\psi f\|_{\dot b^{s,q}_{p,F}}
 \leq C\|f\|_{\dot B^{s,q}_{p,F}(\psi)}.
\end{equation}
Interchanging $\psi$ and $\varphi$ proves the converse inequality
and hence independence of the admissible pair.

The coefficient space is complete, so if $f_\nu$ is Cauchy then
$S_\psi f_\nu$ converges to some $t$ and
\begin{equation}\label{eq:completeness-expanded}
 \|f_\nu-T_\psi t\|_{\dot B^{s,q}_{p,F}}
 =\|T_\psi(S_\psi f_\nu-t)\|_{\dot B^{s,q}_{p,F}}
 \leq C\|S_\psi f_\nu-t\|_{\dot b^{s,q}_{p,F}}\longrightarrow0.
\end{equation} The quasi-norm is definite by the
reproducing formula, and it is a norm when $p,q\geq1$. For
$\phi\in\mathcal S_F$, \eqref{eq:testcoeff} shows that its block
norms decrease faster than any fixed exponential, so
$\mathcal S_F\subset\dot B^{s,q}_{p,F}$ continuously. Finite
synthesis sums belong to $\mathcal S_F$, and when $p,q<\infty$
finite sequences are dense in $\dot b^{s,q}_{p,F}$, which proves
the density assertion.

It remains to prove \eqref{eq:PP}. Let $h$ have the support
box of one bump function and set
\[
 U_h=\sum_R\sup_R|h|\,\mathbf1_R,\qquad
 L_h=\sum_R\inf_R|h|\,\mathbf1_R .
\]
Lemma~\ref{lem:peetre} gives
$\|U_h\|_p\leq C_p\|h\|_p$. Its derivative bound and the
mean value theorem also give, with $\delta=2^{-N}$,
\[
 U_h(z)\leq L_h(z)+C\delta[M_F(|h|^r)(z)]^{1/r}.
\]
Put $\tau=\min(1,p)$. If $h\in L^p$, the scalar maximal
inequality and $|h|\leq U_h$ imply
\[
 \|h\|_p^\tau
 \leq\|L_h\|_p^\tau+(C_p\delta)^\tau\|h\|_p^\tau .
\]
Choose $N_0$ so that $(C_p2^{-N_0})^\tau\leq1/2$ for
the support boxes and spatial exponents in use.
Absorption proves $\|h\|_p\leq C_p\|L_h\|_p$.
For a general band-limited $h$, apply this bound to
$h_\epsilon$ above, noting that $L_{h_\epsilon}\leq L_h$,
and use Fatou's lemma. Thus
$\|U_h\|_p\asymp\|h\|_p\asymp\|L_h\|_p$, also as an
extended-norm assertion. Multiplying by $2^{-js_1}2^{-(j+k)s_2}$ and
taking the outer $\ell^q$ quasi-norm gives the same-pair
comparison, and \eqref{eq:PP} then follows from the independence
of the analysing pair established above.

\section{Scale tails and the little space: proof of Theorem~\ref{thm:little}}
\label{sec:tails}

Fix an admissible pair and a sufficiently fine grid.

\begin{lemma}\label{lem:single-scale}
Let $0<p<\infty$. Uniformly in $j,k$,
\begin{equation}\label{eq:single-scale}
 \left(\sum_{R\in\mathcal R_{j,k}^N}
       \bigg(\sup_R|\psi_{j,k}*f|\bigg)^p |R|\right)^{1/p}
 \leq C\|\psi_{j,k}*f\|_{L^p}.
\end{equation}
\end{lemma}
\begin{proof}
Choose $0<r<\min(1,p)$ and $\lambda r>D$. On each rectangle
$\rho_{a_j,b_{j,k}}$ is uniformly bounded, so
Lemma~\ref{lem:peetre} gives
$\sum_R\sup_R|h|\mathbf1_R\leq C[M_F(|h|^r)]^{1/r}$ for
$h=\psi_{j,k}*f$; Lemma~\ref{lem:maximal} with $p/r>1$ proves
\eqref{eq:single-scale}.
\end{proof}

For $K\geq0$ and $L\geq1$ define the finite Calder\'on sum
\begin{equation}\label{eq:finite-sum}
 f_{K,L}
 =\sum_{|J|_1\leq K}
   \sum_{\substack{R\in\mathcal R_{j,k}^N\\|z_R|\leq L}}
        |R|\,t_{J,R}\psi_{j,k}(\,\cdot-z_R),
\end{equation}
where $t_{J,R}=(\psi_{j,k}*f)(z_R)$. Define the spatial tail
\begin{equation}\label{eq:space-tail}
 \tau_{f,q}(K,L)
 =\left\{\sum_{|J|_1\leq K}\left[
      2^{-js_1}2^{-(j+k)s_2}\left(\sum_{\substack{R\in\mathcal R_{j,k}^N\\|z_R|>L}}
                    |t_{J,R}|^p |R|\right)^{1/p}
                       \right]^q\right\}^{1/q}.
\end{equation}

\begin{proposition}\label{prop:quantitative-approximation}
Let $s\in\mathbb R^2$, $0<p<\infty$ and $0<q\leq\infty$. For
$f\in\dot B^{s,q}_{p,F}$,
\begin{align}
 \|f-f_{K,L}\|_{\dot B^{s,q}_{p,F}}
 &\leq C\bigl(\omega_{f,q}(K)+\tau_{f,q}(K,L)\bigr),
 \label{eq:approximation-error}\\
 \|f_{K,L}\|_{\dot B^{s,q}_{p,F}}
 &\leq C\|f\|_{\dot B^{s,q}_{p,F}}.
 \label{eq:approximation-uniform}
\end{align}
For fixed $K$, $\tau_{f,q}(K,L)\to0$ as $L\to\infty$, and
$f_{K,L}$ converges weakly to $f$ in $\mathcal S_F'$; for finite
$q$ it converges in norm.
\end{proposition}
\begin{proof}
With $E_{K,L}$ the finite set of labels in
\eqref{eq:finite-sum},
write the discarded sequence as
\begin{align}\label{eq:tail-splitting-expanded}
 t-\mathbf1_{E_{K,L}}t&=t^{\mathrm{sc}}+t^{\mathrm{sp}},\notag\\
 t^{\mathrm{sc}}_{J,R}&=\mathbf1_{\{|J|_1>K\}}t_{J,R},\qquad
 t^{\mathrm{sp}}_{J,R}
   =\mathbf1_{\{|J|_1\leq K\}}\mathbf1_{\{|z_R|>L\}}t_{J,R}.
\end{align}
Lemma~\ref{lem:single-scale} and the definition of the spatial
tail imply
\[
 \|t^{\mathrm{sc}}\|_{\dot b^{s,q}_{p,F}}
      \leq C\omega_{f,q}(K),\qquad
 \|t^{\mathrm{sp}}\|_{\dot b^{s,q}_{p,F}}
      =\tau_{f,q}(K,L).
\]
For $\rho=\min(1,p,q)$, with $q$ omitted when $q=\infty$, bounded
synthesis therefore gives
\begin{align}\label{eq:tail-error-expanded}
 \|f-f_{K,L}\|_{\dot B^{s,q}_{p,F}}^\rho
 &=\|T_\psi(t^{\mathrm{sc}}+t^{\mathrm{sp}})\|_{\dot B^{s,q}_{p,F}}^\rho\notag\\
 &\leq C\bigl(\|t^{\mathrm{sc}}\|_{\dot b^{s,q}_{p,F}}^\rho
            +\|t^{\mathrm{sp}}\|_{\dot b^{s,q}_{p,F}}^\rho\bigr)\notag\\
 &\leq C\bigl(\omega_{f,q}(K)^\rho+\tau_{f,q}(K,L)^\rho\bigr).
\end{align}
Taking the $\rho$th root proves \eqref{eq:approximation-error}.
The uniform bound follows from
\[
 \|f_{K,L}\|_{\dot B^{s,q}_{p,F}}
 \leq C\|\mathbf1_{E_{K,L}}S_\psi f\|_{\dot b^{s,q}_{p,F}}
 \leq C\|S_\psi f\|_{\dot b^{s,q}_{p,F}}
 \leq C\|f\|_{\dot B^{s,q}_{p,F}}.
\] For
fixed $K$ there are finitely many scales and each spatial
$p$-sum converges, so its tail tends to zero. The weak convergence follows
from the absolute convergence of synthesis, and for finite $q$
the scale tail of the defining sum tends to zero.
\end{proof}

\begin{remark}\label{rem:rate}
If $\omega_{f,\infty}(K)\leq C_0 2^{-\varepsilon K}$ and
\begin{equation}\label{eq:rate-cond}
 \sup_J 2^{-js_1}2^{-(j+k)s_2}
  \left(\sum_{R\in\mathcal R_{j,k}^N}
      (1+|z_R|)^{\gamma p}|t_{J,R}|^p|R|\right)^{1/p}
       \leq C_1
\end{equation}
for some $\varepsilon,\gamma>0$, then
$\|f-f_{K,L}\|_{\dot B^{s,\infty}_{p,F}}
 \leq C\bigl(C_0 2^{-\varepsilon K}+C_1(1+L)^{-\gamma}\bigr)$.
\end{remark}

We are now ready to prove Theorem~\ref{thm:little}. The limit defining
$d(f)$ exists, since the tail supremum is decreasing. If
$g\in\mathcal S_F$, then $\omega_{g,\infty}(K)\to0$ by
\eqref{eq:testcoeff}; with $\tau=\min(1,p)$, the $\tau$-triangle
inequality at each scale gives
\begin{equation}\label{eq:lower-distance}
 \omega_{f,\infty}(K)^\tau
 \leq \|f-g\|_{\dot B^{s,\infty}_{p,F}}^\tau
                   +\omega_{g,\infty}(K)^\tau.
\end{equation}
Letting $K\to\infty$ and taking the infimum over $g$ proves the
first inequality in \eqref{eq:distance}. Conversely,
$f_{K,L}\in\mathcal S_F$ and \eqref{eq:approximation-error} imply
\begin{align}\label{eq:upper-distance-expanded}
 \inf_{g\in\mathcal S_F}\|f-g\|_{\dot B^{s,\infty}_{p,F}}
 &\leq\liminf_{L\to\infty}\|f-f_{K,L}\|_{\dot B^{s,\infty}_{p,F}}\notag\\
 &\leq C\omega_{f,\infty}(K)\longrightarrow C d(f)
       \qquad(K\to\infty).
\end{align} The distance is zero exactly
on the norm closure, which gives \eqref{eq:little}; the last
assertion follows because the closure is unchanged under
equivalent quasi-norms.

\section{Continuous duality: proofs of Theorems~\ref{thm:duality-all} and~\ref{thm:decomposition}}
\label{sec:duality}

\begin{lemma}\label{lem:sequence-duality}
The norm \eqref{eq:exact-dual} equals
\begin{equation}\label{eq:finite-dual-test}
 \sup\bigl\{|[t,d]_\psi|:\ t\text{ finite support},\
                    \|t\|_{\dot b^{s,q}_{p,F}}\leq1\bigr\}.
\end{equation}
Every $d$ with finite associate norm defines an absolutely
convergent bounded pairing on $\dot b^{s,q}_{p,F}$. If
$q<\infty$, every continuous functional on $\dot b^{s,q}_{p,F}$
is represented uniquely in this way with equality of norms. If
$q=\infty$, every continuous functional on
\begin{equation}\label{eq:little-sequences}
 b_0=\left\{t\in\dot b^{s,\infty}_{p,F}:
       2^{-js_1}2^{-(j+k)s_2}\left(\sum_R|t_{J,R}|^p|R|\right)^{1/p}
                   \longrightarrow0\quad (|J|_1\to\infty)\right\}
\end{equation}
is represented uniquely in this way, again with equality of
norms.
\end{lemma}
\begin{proof}
The change of variables
\begin{equation}\label{eq:change-variables}
 v_{J,R}=2^{-js_1}2^{-(j+k)s_2}|R|^{1/p}t_{J,R},\qquad
 e_{J,R}=2^{js_1}2^{(j+k)s_2}|R|^{1-1/p}d_{J,R}
\end{equation}
turns the primal norm into $\ell^q(\ell^p)$, the proposed dual
norm into $\ell^{q^*}(\ell^{p^*})$, and the pairing into
$\sum_{J,R}v_{J,R}\overline{e_{J,R}}$. For $p>1$, H\"older
within each block gives the upper estimate. For $p\leq1$,
\begin{equation}\label{eq:small-p-block}
 \sum_R|v_{J,R}e_{J,R}|
 \leq\sup_R|e_{J,R}|
                    \left(\sum_R|v_{J,R}|^p\right)^{1/p}.
\end{equation}
Set $A_J=\|v_{J,\cdot}\|_{\ell^p}$ and
$E_J=\|e_{J,\cdot}\|_{\ell^{p^*}}$. The outer estimate is
\begin{equation}\label{eq:outer-duality-expanded}
 \sum_{J,R}|v_{J,R}e_{J,R}|
 \leq\sum_J A_JE_J
 \leq\|\{A_J\}\|_{\ell^q}\|\{E_J\}\|_{\ell^{q^*}}.
\end{equation}
For $q\leq1$, the last step is
$\sum_J A_J\leq(\sum_J A_J^q)^{1/q}$; the other cases are
H\"older's inequality, including $\ell^\infty$--$\ell^1$.

We also give the reverse estimate explicitly. Work first on a
finite set of coordinates and write $\operatorname{sgn}z=z/|z|$
for $z\ne0$, with $\operatorname{sgn}0=0$. For each nonzero block
choose a vector $b_J$ by
\begin{equation}\label{eq:block-maximizer-expanded}
 b_{J,R}=
 \begin{cases}
 \operatorname{sgn}(e_{J,R})|e_{J,R}|^{p'-1}/E_J^{p'-1},&p>1,\\
 \operatorname{sgn}(e_{J,R_J})\mathbf1_{\{R=R_J\}},&p\leq1,
 \end{cases}
\end{equation}
where $|e_{J,R_J}|=E_J$ in the second case. Then
\[
 \|b_J\|_{\ell^p}=1,\qquad
 \sum_R b_{J,R}\overline{e_{J,R}}=E_J.
\]
Zero blocks are assigned $b_J=0$. For $1<q<\infty$ take
$a_J=E_J^{q'-1}/\|E\|_{\ell^{q'}}^{q'-1}$; for $q\leq1$ take
$a_J=\mathbf1_{\{J=J_*\}}$ at a largest block; for $q=\infty$
take $a_J=1$ on the finite set of nonzero blocks. In every case,
\begin{equation}\label{eq:double-maximizer-expanded}
 v_{J,R}=a_Jb_{J,R},\qquad
 \|v\|_{\ell^q(\ell^p)}=1,\qquad
 \sum_{J,R}v_{J,R}\overline{e_{J,R}}
                  =\|E\|_{\ell^{q^*}}.
\end{equation}
The zero sequence is immediate. Increasing the finite sets
proves \eqref{eq:finite-dual-test}.

If $\Lambda$ is a bounded functional and $\mathbf e^{J,R}$
is the sequence equal to one at $(J,R)$ and zero elsewhere, set
\begin{equation}\label{eq:coordinate-functional-expanded}
 d_{J,R}=|R|^{-1}\overline{\Lambda(\mathbf e^{J,R})}.
\end{equation}
For every finitely supported $t$, linearity and the finite-set
identity give
\[
 \Lambda(t)=[t,d]_\psi,\qquad
 \|d\|_{(\dot b^{s,q}_{p,F})^\times}
 =\sup_{\substack{t\text{ finite}\\\|t\|\leq1}}|\Lambda(t)|
 \leq\|\Lambda\|.
\] If $q<\infty$, finite
sequences are dense and the representation extends. If
$q=\infty$, finite sequences are dense precisely in $b_0$, using
the vanishing scale tail and then $p<\infty$ within the remaining
blocks.
\end{proof}

\begin{proposition}\label{prop:associate-retraction}
The synthesis series $T_\psi d$ is well defined in
$\mathcal S_F'$ for every
$d\in(\dot b^{s,q}_{p,F})^\times$, with absolute pairings
against test functions, and
\begin{equation}\label{eq:associate-synthesis}
 \|T_\psi d\|_{(\dot B^{s,q}_{p,F})^\times}
                 \leq C\|d\|_{(\dot b^{s,q}_{p,F})^\times}.
\end{equation}
The distribution space \eqref{eq:distribution-dual} is Banach
and independent of the admissible pair and the fixed grid. For
$f\in\dot B^{s,q}_{p,F}$ and
$g\in(\dot B^{s,q}_{p,F})^\times$, the pairing
\eqref{eq:duality-pairing} converges absolutely, agrees with the
distribution pairing when either entry is a test function, and is
independent of both choices, including when $q=\infty$.
\end{proposition}
\begin{proof}
Every associate sequence satisfies
\begin{equation}\label{eq:dual-growth}
 |d_{J,R}|\leq 2^{-js_1}2^{-(j+k)s_2}|R|^{1/p-1}
                   \|d\|_{(\dot b^{s,q}_{p,F})^\times},
\end{equation}
which, together with \eqref{eq:testcoeff} and the lattice
counting bound, gives the absolute convergence of the synthesis
series against test functions exactly as in the proof of
Theorem~\ref{thm:retraction}. Set $P_\psi=S_\psi T_\psi$. For
finitely supported sequences the symmetry of the real even
kernels gives
\begin{equation}\label{eq:transpose}
                  [t,P_\psi d]_\psi=[P_\psi t,d]_\psi.
\end{equation}
The kernel identity behind this transposition is
\begin{align}\label{eq:transpose-matrix-expanded}
 (P_\psi d)_{I,Q}
 &=\sum_{J,R}|R|\,d_{J,R}
          (\psi_I*\psi_J)(z_Q-z_R),\notag\\
 [t,P_\psi d]_\psi
 &=\sum_{I,Q,J,R}|Q||R|\,t_{I,Q}\overline{d_{J,R}}
          (\psi_I*\psi_J)(z_Q-z_R)\notag\\
 &=[P_\psi t,d]_\psi.
\end{align}
Thus Lemma~\ref{lem:sequence-duality} gives
\begin{align}\label{eq:transpose-bound-expanded}
 \|P_\psi d\|_\times
 &=\sup_{\substack{t\text{ finite}\\\|t\|_{\dot b^{s,q}_{p,F}}\leq1}}
                           |[t,P_\psi d]_\psi|\notag\\
 &\leq\sup_{\substack{t\text{ finite}\\\|t\|_{\dot b^{s,q}_{p,F}}\leq1}}
            \|P_\psi t\|_{\dot b^{s,q}_{p,F}}\|d\|_\times
 \leq C\|d\|_\times.
\end{align}
The calculation starts with finite restrictions of $d$;
Fatou's lemma and the absolute distributional convergence
then give the estimate for arbitrary associate sequences. Since
$P_\psi^2=P_\psi$ and $T_\psi S_\psi=I$, the range of $P_\psi$
is identified with \eqref{eq:distribution-dual}, which is
therefore complete and definite.

For another admissible pair $\varphi$, the cross matrix
$S_\varphi T_\psi$ is transposed to $S_\psi T_\varphi$ under the
volume pairings; the finite-sequence argument bounds both
cross maps on the associate spaces, and applying this to
$g=T_\psi S_\psi g=T_\varphi S_\varphi g$ gives the independence
of norms. For the full $q=\infty$ range the estimates in
Lemma~\ref{lem:grid} use absolute values throughout, so the
positive cross matrix is bounded on the primal spaces as well;
hence for primal $t$ and associate $d$,
\begin{equation}\label{eq:cross-abs}
 \sum_{I,Q}\sum_{J,R}|Q||R|\,|t_{J,R}|\,|d_{I,Q}|\,
       |(\varphi_I*\psi_J)(z_Q-z_R)|
 \leq C\|t\|_{\dot b^{s,q}_{p,F}(\psi)}
       \|d\|_{(\dot b^{s,q}_{p,F}(\varphi))^\times},
\end{equation}
which justifies the cross transpose identity for arbitrary
sequences. Agreement with the distribution pairing follows by
pairing the exact formula, which converges in $\mathcal S_F$ for
test entries.
\end{proof}

We now give the proof of Theorem~\ref{thm:duality-all}. An element of the
associate space defines a bounded functional by
Proposition~\ref{prop:associate-retraction}. Conversely, for
$q<\infty$ and $\ell\in(\dot B^{s,q}_{p,F})^*$, the composition
$\ell\circ T_\psi$ is a bounded functional on
$\dot b^{s,q}_{p,F}$; Lemma~\ref{lem:sequence-duality} gives $d$
with $\ell(T_\psi t)=[t,d]_\psi$ and
$\|d\|_\times\leq C\|\ell\|$. Set $g=T_\psi d$. On test
functions,
\begin{equation}\label{eq:represent}
 \langle f,g\rangle_F=[S_\psi f,d]_\psi
           =\ell(T_\psi S_\psi f)=\ell(f);
\end{equation}
the density of $\mathcal S_F$ extends this to the whole space,
which proves \eqref{eq:full-finite-dual}. At $q=\infty$,
Theorem~\ref{thm:little} and Lemma~\ref{lem:single-scale} give
$S_\psi B_0\subset b_0$, while finite sequences are dense in
$b_0$ and bounded synthesis gives $T_\psi b_0\subset B_0$; the
same argument therefore proves \eqref{eq:little-dual}.
In both cases the two norm estimates are
\begin{equation}\label{eq:dual-norms-expanded}
 \|\ell_g\|\leq C\|g\|_{(\dot B^{s,q}_{p,F})^\times},\qquad
 \|g\|_{(\dot B^{s,q}_{p,F})^\times}
 \leq C\|d\|_{(\dot b^{s,q}_{p,F})^\times}
 \leq C\|\ell\|.
\end{equation}
Two representatives agreeing on $\mathcal S_F$ are equal in
$\mathcal S_F'$, which proves uniqueness.

We next prove Theorem~\ref{thm:decomposition}. Restricting $\ell$
to $B_0$ gives a unique $g\in B^\times$ with
$\|g\|_{B^\times}\leq C\|\ell\|$, and $\ell_g$ is defined on all
of $B$ by Proposition~\ref{prop:associate-retraction}. Then
$\ell_0=\ell-\ell_g$ vanishes on $B_0$, and
\begin{align}\label{eq:decomposition-bounds-expanded}
 \|\ell_0\|&\leq\|\ell\|+\|\ell_g\|
             \leq(1+C)\|\ell\|,\notag\\
 \|\ell\|&\leq\|\ell_g\|+\|\ell_0\|
             \leq C\|g\|_{B^\times}+\|\ell_0\|.
\end{align}
These are the two directions of \eqref{eq:decomp-norm}.
If $\ell=\ell_h+L_0$ is a second decomposition, restriction gives
\[
 \ell_{g-h}|_{B_0}=0\quad\Longrightarrow\quad g=h
 \quad\Longrightarrow\quad \ell_0=L_0.
\]
The implication uses the uniqueness already proved for $B_0^*$.
Frame independence follows from that of the canonical pairing.

For an example with a nonzero annihilator, take $p=2$ and $s=0$.
Choose a nonzero $\rho\in C_c^\infty$ supported in a small region
with $\eta\ne0$ on which $|\widehat\psi_{J_0}|\geq c>0$. For a
sufficiently large integer $L$ put
\begin{equation}\label{eq:example}
 \widehat h_\nu(\zeta)=2^{-L\nu D/2}\rho(2^{-L\nu}\zeta),
 \qquad \nu\geq1,\qquad f=\sum_{\nu\geq1}h_\nu.
\end{equation}
Choose $L$ so that each joint-frequency annulus meets at most
one of these supports. They escape to infinity, so their sum
defines a tempered distribution. Under our Fourier convention,
Plancherel and the substitution $\zeta=2^{L\nu}\omega$ give
\begin{align}\label{eq:counterexample-normalization-expanded}
 \|h_\nu\|_2^2
 &=(2\pi)^{-D}2^{-L\nu D}
               \int|\rho(2^{-L\nu}\zeta)|^2\,d\zeta
   =(2\pi)^{-D}\|\rho\|_2^2,\notag\\
 \sup_J\|\psi_J*f\|_2&\leq C\|\rho\|_2,\notag\\
 \|\psi_{J_0+(L\nu,L\nu)}*f\|_2^2
 &=(2\pi)^{-D}\int
       |\widehat\psi_{J_0}(\omega)|^2|\rho(\omega)|^2\,d\omega
   \geq c^2(2\pi)^{-D}\|\rho\|_2^2.
\end{align}
Consequently $f\in B$ but $d(f)>0$, hence $f\notin B_0$.
The quotient $B/B_0$ is a nonzero Banach space. Hahn--Banach
provides $\Lambda\in(B/B_0)^*$ with $\Lambda(f+B_0)\ne0$; the
functional $L_0(h)=\Lambda(h+B_0)$ belongs to $B_0^\perp$ and is
nonzero.

\begin{proof}[Proof of Corollary~\ref{cor:duality}]
For $p>1$,
$D_J(d)=(\sum_R|R||d_{J,R}|^{p'})^{1/p'}$, so the associate
coefficient norm is exactly $\dot b^{-s,q^*}_{p',F}$; the
discrete characterisation identifies its distribution space
with $\dot B^{-s,q^*}_{p',F}$, and \eqref{eq:little-dual-p}
follows from the outer $\ell^1$ norm.
\end{proof}

\begin{proof}[Proof of Corollary~\ref{cor:small-p-duality}]
For a nonzero bump function, $|R|\asymp_N2^{-jn}2^{-km}$, so with
$1-1/p=-\delta$,
\begin{equation}\label{eq:weight-identity}
 2^{js_1}2^{(j+k)s_2}|R|^{1-1/p}
 \asymp
 2^{j(s_1+s_2+n\delta)}2^{k(s_2+m\delta)}.
\end{equation}
Matching the coefficients of $j$ and $k$ gives
\begin{equation}\label{eq:small-p-linear-system-expanded}
 -(t_1+t_2)=s_1+s_2+n\delta,\qquad
 -t_2=s_2+m\delta,
\end{equation}
and therefore
\[
 t_2=-s_2-m\delta,\qquad
 t_1=-s_1+(m-n)\delta.
\]
Thus the right-hand side of \eqref{eq:weight-identity} is
$2^{-jt_1}2^{-(j+k)t_2}$ for the stated $t$. By
Lemma~\ref{lem:Linfty-sampling},
\begin{equation}\label{eq:small-p-associate}
 \|S_\psi g\|_{(\dot b^{s,q}_{p,F})^\times}
          \asymp\|g\|_{\dot B^{t,q^*}_{\infty,F}},
\end{equation}
and Theorem~\ref{thm:duality-all} gives the asserted duals.
\end{proof}

\section{Interpolation: proofs of Theorems~\ref{thm:real-interpolation} and~\ref{thm:complex-interpolation}}
\label{sec:interpolation}

All couples are compatible in $\mathcal S_F'$. We use one
admissible pair and one grid fine enough for every parameter in
the stated identities, so the analysis and synthesis maps are
the same at both endpoints and $T_\psi S_\psi=I$ on their sum.

For a compatible quasi-Banach couple $(X_0,X_1)$,
$K(t,x;X_0,X_1)=\inf_{x=x_0+x_1}(\|x_0\|_{X_0}+t\|x_1\|_{X_1})$,
and the real interpolation quasi-norm is
$\|x\|_{(X_0,X_1)_{\theta,r}}
 =(\int_0^\infty[t^{-\theta}K(t,x)]^r\,dt/t)^{1/r}$.

We first prove Theorem~\ref{thm:real-interpolation}. For a
labelled sequence $t$ put
\[
 u_J=2^{-js_1}2^{-(j+k)s_2}
        \left(\sum_R|t_{J,R}|^p|R|\right)^{1/p},
 \qquad E_i=\dot b^{s,q_i}_{p,F}.
\]
We claim
\begin{equation}\label{eq:block-K}
 K(\lambda,t;E_0,E_1)
        \asymp K(\lambda,u;\ell^{q_0},\ell^{q_1}).
\end{equation}
For a nonnegative sequence $u$, the scalar $K$-functional can
be computed using positive decompositions $u=a+b$. Indeed, any
decomposition $u=c+d$ can be replaced, coordinatewise, by
$a=u|c|/(|c|+|d|)$ and $b=u|d|/(|c|+|d|)$, so that
$a\leq|c|$ and $b\leq|d|$. Given $u=a+b$, define
\begin{equation}\label{eq:block-K-decomposition-expanded}
 t^0_{J,R}=\frac{a_J}{u_J}t_{J,R},\qquad
 t^1_{J,R}=\frac{b_J}{u_J}t_{J,R},
\end{equation}
with both terms zero when $u_J=0$. By homogeneity in each block,
\[
 \|t^0\|_{E_0}=\|a\|_{\ell^{q_0}},\qquad
 \|t^1\|_{E_1}=\|b\|_{\ell^{q_1}},
\]
and hence
\begin{equation}\label{eq:block-K-upper-expanded}
 K(\lambda,t;E_0,E_1)
 \leq\inf_{u=a+b}\bigl(\|a\|_{\ell^{q_0}}
                      +\lambda\|b\|_{\ell^{q_1}}\bigr)
 =K(\lambda,u;\ell^{q_0},\ell^{q_1}).
\end{equation}
Conversely, for $t=t^0+t^1$ let $v_J,z_J$ be the corresponding
weighted block norms. The quasi-triangle inequality gives
$u_J\leq C_p(v_J+z_J)$. Therefore
\[
 a_J=\frac{u_Jv_J}{v_J+z_J}\leq C_pv_J,\qquad
 b_J=\frac{u_Jz_J}{v_J+z_J}\leq C_pz_J,\qquad u=a+b,
\]
where zero denominators are assigned zero. It follows that
\begin{align}\label{eq:block-K-lower-expanded}
 K(\lambda,u;\ell^{q_0},\ell^{q_1})
 &\leq\|a\|_{\ell^{q_0}}+\lambda\|b\|_{\ell^{q_1}}\notag\\
 &\leq C_p\bigl(\|t^0\|_{E_0}+\lambda\|t^1\|_{E_1}\bigr).
\end{align}
Taking the infimum over all decompositions of $t$ proves
\eqref{eq:block-K}.

The scalar sequence identity
\begin{equation}\label{eq:scalar-real}
             (\ell^{q_0},\ell^{q_1})_{\theta,q}=\ell^q,
 \qquad \frac1q=\frac{1-\theta}{q_0}+\frac\theta{q_1},
\end{equation}
can be read directly from the rearrangement formula. Write $u^*$
for the decreasing rearrangement with respect to counting measure,
and put $\kappa=(1/q_0-1/q_1)^{-1}$. For $q_1<\infty$,
Holmstedt's formula \cite{Holmstedt70}, also discussed in
\cite[Chapter~5]{BL76}, gives
\begin{align}\label{eq:holmstedt-expanded}
 K(\lambda,u;\ell^{q_0},\ell^{q_1})
 &\asymp\left(\int_0^{\lambda^\kappa}u^*(v)^{q_0}\,dv\right)^{1/q_0}
 \notag\\
 &\quad+\lambda\left(\int_{\lambda^\kappa}^\infty
                         u^*(v)^{q_1}\,dv\right)^{1/q_1}.
\end{align}
When $q_1=\infty$, this becomes
\begin{equation}\label{eq:holmstedt-supremum-expanded}
 K(\lambda,u;\ell^{q_0},\ell^\infty)
 \asymp\left(\int_0^{\lambda^{q_0}}u^*(v)^{q_0}\,dv\right)^{1/q_0}.
\end{equation}
The corresponding Hardy inequalities and the displayed relation
between $q,q_0,q_1$ yield
\begin{align}\label{eq:real-rearrangement-integral-expanded}
 \int_0^\infty[\lambda^{-\theta}K(\lambda,u)]^q
                         \frac{d\lambda}{\lambda}
 &\asymp\int_0^\infty[v^{1/q}u^*(v)]^q\frac{dv}{v}\notag\\
 &=\int_0^\infty u^*(v)^q\,dv=\sum_Ju_J^q.
\end{align}
These formulas also apply when $q_0<1$, and prove
\eqref{eq:scalar-real} with equivalent quasi-norms. Together
with \eqref{eq:block-K}, they give
$(E_0,E_1)_{\theta,q}=\dot b^{s,q}_{p,F}$.

The endpoint bounds on $S_\psi,T_\psi$ give
\begin{equation}\label{eq:retraction-K}
 K(\lambda,S_\psi f;E_0,E_1)
                  \asymp K(\lambda,f;B_0,B_1),
\end{equation}
where $B_i=\dot B^{s,q_i}_{p,F}$. More explicitly,
\begin{align}\label{eq:retraction-K-bounds-expanded}
 K(\lambda,S_\psi f;E_0,E_1)
 &\leq\max_{i=0,1}\|S_\psi\|_{B_i\to E_i}
                      K(\lambda,f;B_0,B_1),\notag\\
 K(\lambda,f;B_0,B_1)
 &\leq\max_{i=0,1}\|T_\psi\|_{E_i\to B_i}
                      K(\lambda,S_\psi f;E_0,E_1).
\end{align}
The first line analyses a decomposition of $f$; the second
synthesises a decomposition of $S_\psi f$, using $T_\psi S_\psi f=f$.
Thus
\begin{align}\label{eq:real-norm-chain-expanded}
 \|f\|_{(B_0,B_1)_{\theta,q}}
 &\asymp\|S_\psi f\|_{(E_0,E_1)_{\theta,q}}
 \asymp\|S_\psi f\|_{\dot b^{s,q}_{p,F}}
 \asymp\|f\|_{\dot B^{s,q}_{p,F}},
\end{align}
which proves \eqref{eq:real-interpolation}.

We now prove Theorem~\ref{thm:complex-interpolation}. The
coefficient spaces satisfy
\begin{equation}\label{eq:complex-sequence}
 [\dot b^{s^0,q_0}_{p_0,F},\dot b^{s^1,q_1}_{p_1,F}]_\theta
 =\dot b^{s,q}_{p,F}
\end{equation}
with the parameters in \eqref{eq:interpolated-parameters}.
We give the analytic family in the mixed-sequence version of
the Calder\'on construction; see \cite[Sections 5.1 and 5.6]{BL76}.
On the strip $0\leq\Re z\leq1$, put
\begin{equation}\label{eq:complex-family}
 \frac1{p(z)}=\frac{1-z}{p_0}+\frac z{p_1},\quad
 \frac1{q(z)}=\frac{1-z}{q_0}+\frac z{q_1},
\end{equation}
and define
\begin{equation}\label{eq:complex-weight}
 w_z(J)=
 \bigl(2^{-js_1^0}2^{-(j+k)s_2^0}\bigr)^{1-z}
 \bigl(2^{-js_1^1}2^{-(j+k)s_2^1}\bigr)^z.
\end{equation}
Take a finite sequence $t$ with $\|t\|_{\dot b^{s,q}_{p,F}}=1$, and set
\begin{equation}\label{eq:complex-polar-data-expanded}
 v_{J,R}=w_\theta(J)|R|^{1/p}t_{J,R},\qquad
 A_J=\left(\sum_R|v_{J,R}|^p\right)^{1/p},\qquad
 \sum_J A_J^q=1.
\end{equation}
On nonzero coordinates define
\begin{equation}\label{eq:analytic-coefficient-family-expanded}
 F_{J,R}(z)=w_z(J)^{-1}|R|^{-1/p(z)}
       \operatorname{sgn}(v_{J,R})
       |v_{J,R}|^{p/p(z)}A_J^{q/q(z)-p/p(z)},
\end{equation}
and set $F_{J,R}=0$ on zero coordinates. Positive bases are
raised to complex powers using their real logarithms. Then
$F(\theta)=t$, and for $i=0,1$ and $y\in\mathbb R$,
\begin{align}\label{eq:analytic-boundary-block-expanded}
 &w_i(J)\left(\sum_R|F_{J,R}(i+\mathrm i y)|^{p_i}|R|\right)^{1/p_i}\notag\\
 &\quad=\left(\sum_R|v_{J,R}|^p
                   A_J^{p_iq/q_i-p}\right)^{1/p_i}
       =A_J^{q/q_i},\notag\\
 &\|F(i+\mathrm i y)\|_{\dot b^{s^i,q_i}_{p_i,F}}
       =\left(\sum_J A_J^q\right)^{1/q_i}=1.
\end{align}
If the interpolation class requires decay as $|\Im z|\to\infty$,
replace $F(z)$ by $e^{\varepsilon(z-\theta)^2}F(z)$ and then let
$\varepsilon\downarrow0$. Thus
\[
 \|t\|_{[\dot b^{s^0,q_0}_{p_0,F},\dot b^{s^1,q_1}_{p_1,F}]_\theta}
 \leq\|t\|_{\dot b^{s,q}_{p,F}}.
\]

For the reverse bound, write $E_i=\dot b^{s^i,q_i}_{p_i,F}$ and
$E=\dot b^{s,q}_{p,F}$. Let $F$ be any admissible analytic family
with $F(\theta)=t$. For a finite $d\in E^\times$ with
$\|d\|_{E^\times}\leq1$, apply the same construction to the dual
parameters $p_i',q_i',-s^i$, obtaining $D(\theta)=d$ and
$\sup_y\|D(i+\mathrm i y)\|_{E_i^\times}\leq1$ (up to an arbitrarily
small damping factor). The scalar function
\[
 \Phi(z)=\sum_{J,R}|R|F_{J,R}(z)
                           \overline{D_{J,R}(\bar z)}
\]
is analytic because $D$ has finite support. On the boundary,
\begin{equation}\label{eq:complex-dual-boundary-expanded}
 |\Phi(i+\mathrm i y)|
 \leq\|F(i+\mathrm i y)\|_{E_i}\|D(i-\mathrm i y)\|_{E_i^\times}
 \leq\|F(i+\mathrm i y)\|_{E_i}.
\end{equation}
The three-lines theorem and finite-sequence duality yield
\begin{align}\label{eq:complex-reverse-expanded}
 \|t\|_E
 &=\sup_{\substack{d\text{ finite}\\\|d\|_{E^\times}\leq1}}
                      |[t,d]_\psi|\notag\\
 &\leq\left(\sup_y\|F(\mathrm i y)\|_{E_0}\right)^{1-\theta}
       \left(\sup_y\|F(1+\mathrm i y)\|_{E_1}\right)^\theta.
\end{align}
Taking the infimum over $F$ proves the reverse estimate.
Finite sequences are dense in $E$; applying the first estimate
to differences of truncations extends it to every $t\in E$.
This proves \eqref{eq:complex-sequence}.

The fixed linear maps $S_\psi,T_\psi$ act boundedly between the
complex interpolation spaces by their endpoint bounds. If
$f\in[B_0,B_1]_\theta$, then
$S_\psi f\in\dot b^{s,q}_{p,F}$ by \eqref{eq:complex-sequence},
and $T_\psi S_\psi f=f$ gives
$\|f\|_{\dot B^{s,q}_{p,F}}\leq C\|f\|_{[B_0,B_1]_\theta}$.
Conversely, if $f\in\dot B^{s,q}_{p,F}$, its analysis belongs
to $[\dot b^{s^0,q_0}_{p_0,F},\dot b^{s^1,q_1}_{p_1,F}]_\theta$,
and synthesis gives the opposite bound.

\section{Lifting: proof of Theorem~\ref{thm:lifting}}
\label{sec:lifting}

Multiplication by $m_\alpha$ and $m_{-\alpha}$ is continuous on
the Fourier transforms of $\mathcal S_F$: these are Schwartz and
flat on $\eta=0$ by \eqref{eq:partialmoments}, so the fixed power
singularities of the multiplier are absorbed by Taylor's formula
and its polynomial growth by Schwartz decay. Thus $\mathcal D^\alpha$
is defined on $\mathcal S_F'$ by transposition, and
$m_\alpha m_{-\alpha}=1$ away from $\eta=0$.

Let $\psi,\varphi$ be admissible pairs, $J=(j,k)$, $L=(u,v)$, and
put
\begin{equation}\label{eq:lifting-kernel-def}
 H^\alpha_{J,L}=\mathcal F^{-1}
                [\widehat\psi_J m_\alpha\widehat\varphi_L].
\end{equation}

\begin{lemma}\label{lem:lifting-localisation}
The kernel $H^\alpha_{J,L}$ vanishes unless
$|j-u|+|k-v|\leq C_0$. With $A=\max(a_j,a_u)$ and
$B=\max(b_{j,k},b_{u,v})$, for every $M,E>0$,
\begin{equation}\label{eq:lifting-localisation}
 |H^\alpha_{J,L}(x,y)|
 \leq C\,2^{j(\alpha_1+\alpha_2)+k\alpha_2}
       2^{-E(|j-u|+|k-v|)} P_{A,B,M}(x,y).
\end{equation}
\end{lemma}
\begin{proof}
On a nonempty intersection the annular supports give finite
overlap and $b_{j,k}\asymp2^{-k}$. Consider the normalised symbol
\begin{equation}\label{eq:normalised-symbol}
 G_{J,L}(\zeta,\tau)
 =2^{-j(\alpha_1+\alpha_2)-k\alpha_2}
 \widehat\psi_J(\zeta/a,\tau/b)
 m_\alpha(\zeta/a,\tau/b)
 \widehat\varphi_L(\zeta/a,\tau/b),
\end{equation}
where $a=a_j$, $b=b_{j,k}$. The homogeneity of $m_\alpha$ and
the annular restrictions give, for all multi-indices,
\begin{equation}\label{eq:symbol-derivatives}
 |\partial_\xi^\gamma\partial_\eta^\delta m_\alpha(\xi,\eta)|
 \leq C\,2^{j(\alpha_1+\alpha_2)+k\alpha_2}
       2^{-j|\gamma|}2^{-k|\delta|},
\end{equation}
including for negative powers, since the annular norms are bounded
away from zero. The change of variables, including the cutoff derivatives,
therefore gives
\begin{equation}\label{eq:normalized-derivative-uniform-expanded}
 \sup_{J,L}\|\partial_\zeta^\gamma\partial_\tau^\delta
                      G_{J,L}\|_{L^\infty}<\infty,
 \qquad\operatorname{supp}G_{J,L}\subset K,
\end{equation}
where $K$ is one fixed compact set. Writing
$g_{J,L}=\mathcal F^{-1}G_{J,L}$ and integrating by parts separately
in the two frequency variables, we obtain
\begin{equation}\label{eq:inverse-symbol-decay-expanded}
 |g_{J,L}(X,Y)|
 \leq C_{M_1,M_2}(1+|X|)^{-M_1}(1+|Y|)^{-M_2}.
\end{equation}
The Fourier scaling now gives
\begin{align}\label{eq:lifting-kernel-scaling-expanded}
 H^\alpha_{J,L}(x,y)
 &=2^{j(\alpha_1+\alpha_2)+k\alpha_2}
                       a^{-n}b^{-m}g_{J,L}(x/a,y/b),\notag\\
 |H^\alpha_{J,L}(x,y)|
 &\leq C\,2^{j(\alpha_1+\alpha_2)+k\alpha_2}
       \frac{a^M}{(a+|x|)^{n+M}}
       \frac{b^M}{(b+|y|)^{m+M}}.
\end{align}
On the finite overlap $a\asymp A$ and $b\asymp B$.
Inserting the factor $2^{-E(|j-u|+|k-v|)}$, with a changed
constant, proves \eqref{eq:lifting-localisation}.
\end{proof}

We now prove Theorem~\ref{thm:lifting}. Applying $\psi_J$ to the
exact $\varphi$ reproducing formula gives
\begin{equation}\label{eq:lifting-action}
 \psi_J*\mathcal D^\alpha f(z)
 =\sum_{|J-L|_1\leq C_0}
     \sum_{R\in\mathcal R_L^N}
        |R|(\varphi_L*f)(z_R)H^\alpha_{J,L}(z-z_R).
\end{equation}
Multiplying the kernel estimate by the weight at smoothness
$s+\alpha$, the multiplier factor cancels exactly:
\begin{equation}\label{eq:lifting-cancel}
 2^{-j(s_1+\alpha_1)}2^{-(j+k)(s_2+\alpha_2)}
       2^{j(\alpha_1+\alpha_2)+k\alpha_2}
       =2^{-js_1}2^{-(j+k)s_2}.
\end{equation}
Lemma~\ref{lem:matrix} and the discrete characterisation give
\begin{equation}\label{eq:lifting-bound}
 \|\mathcal D^\alpha f\|_{\dot B^{s+\alpha,q}_{p,F}}
 \leq C\left\|\left\{2^{-us_1}2^{-(u+v)s_2}
       \sum_{R\in\mathcal R_L^N}
               |(\varphi_L*f)(z_R)|\mathbf1_R\right\}_L
                                             \right\|_{\ell^q(L^p)}
 \leq C\|f\|_{\dot B^{s,q}_{p,F}}.
\end{equation}
Apply \eqref{eq:lifting-bound} with $-\alpha$ and smoothness
$s+\alpha$. Since $\mathcal D^{-\alpha}\mathcal D^\alpha=I$
on $\mathcal S_F'$, we obtain
\begin{align}\label{eq:lifting-inverse-expanded}
 \|f\|_{\dot B^{s,q}_{p,F}}
 &=\|\mathcal D^{-\alpha}\mathcal D^\alpha f\|_{\dot B^{s,q}_{p,F}}
 \leq C\|\mathcal D^\alpha f\|_{\dot B^{s+\alpha,q}_{p,F}},\notag\\
 h\in\dot B^{s+\alpha,q}_{p,F}
 &\quad\Longrightarrow\quad f=\mathcal D^{-\alpha}h\in\dot B^{s,q}_{p,F},
 \qquad\mathcal D^\alpha f=h.
\end{align}
For the little spaces, take $f_\nu\in\mathcal S_F$ with
$\|f-f_\nu\|_{\dot B^{s,\infty}_{p,F}}\to0$. Then
\[
 \mathcal D^\alpha f_\nu\in\mathcal S_F,\qquad
 \|\mathcal D^\alpha f-\mathcal D^\alpha f_\nu\|_{\dot B^{s+\alpha,\infty}_{p,F}}
 \leq C\|f-f_\nu\|_{\dot B^{s,\infty}_{p,F}}\longrightarrow0.
\]
The argument with $-\alpha$ gives the reverse inclusion.

\appendix
\section{Mixed-norm comparisons and Hardy-space embeddings}
\label{app:mixed}

Use one admissible pair and set
$G_J=2^{-js_1}2^{-(j+k)s_2}|\psi_{j,k}*f|$. The flag Triebel--Lizorkin norm is
\begin{equation}\label{eq:Fdefinition}
 \|f\|_{\dot F^{s,q}_{p,F}}=\|\{G_J\}\|_{L^p(\ell^q)}.
\end{equation}

\begin{proposition}\label{prop:FB-embeddings}
Let $s\in\mathbb R^2$, $0<p<\infty$ and $0<q\leq\infty$. Then
\begin{equation}\label{eq:FB-embeddings}
 \dot B^{s,\min(p,q)}_{p,F}
       \hookrightarrow\dot F^{s,q}_{p,F}
       \hookrightarrow\dot B^{s,\max(p,q)}_{p,F},
\end{equation}
and $\dot F^{s,p}_{p,F}=\dot B^{s,p}_{p,F}$. Moreover, for
$0<q_0\leq q_1\leq\infty$,
$\dot F^{s,q_0}_{p,F}\hookrightarrow\dot F^{s,q_1}_{p,F}$ and
$\dot B^{s,q_0}_{p,F}\hookrightarrow\dot B^{s,q_1}_{p,F}$.
\end{proposition}
\begin{proof}
If $q\leq p$, Minkowski's inequality in $L^{p/q}$ and the
pointwise inclusion $\ell^q\subset\ell^p$ give
\begin{align}\label{eq:mixed-q-small-expanded}
 \|\{G_J\}\|_{L^p(\ell^q)}^q
 &=\left\|\sum_JG_J^q\right\|_{L^{p/q}}
 \leq\sum_J\|G_J^q\|_{L^{p/q}}
 =\sum_J\|G_J\|_{L^p}^q,\notag\\
 \sum_J\|G_J\|_{L^p}^p
 &=\int\sum_JG_J^p
 \leq\int\left(\sum_JG_J^q\right)^{p/q}
 =\|\{G_J\}\|_{L^p(\ell^q)}^p.
\end{align}
These are the two embeddings for $q\leq p$. If $p\leq q<\infty$,
the pointwise inclusion $\ell^p\subset\ell^q$ and Minkowski's
integral inequality in $\ell^{q/p}$ give
\begin{align}\label{eq:mixed-p-small-expanded}
 \|\{G_J\}\|_{L^p(\ell^q)}^p
 &=\int\left(\sum_JG_J^q\right)^{p/q}
 \leq\int\sum_JG_J^p
 =\sum_J\|G_J\|_{L^p}^p,\notag\\
 \left\|\{\|G_J\|_{L^p}\}\right\|_{\ell^q}^p
 &=\left\|\left\{\int G_J^p\right\}_J\right\|_{\ell^{q/p}}\notag\\
 &\leq\int\|\{G_J(z)^p\}\|_{\ell^{q/p}}\,dz
 =\|\{G_J\}\|_{L^p(\ell^q)}^p.
\end{align}
At $q=\infty$ the required estimates are
\begin{equation}\label{eq:mixed-infinity-expanded}
 \sup_J\int G_J^p
 \leq\int\sup_JG_J^p
 \leq\int\sum_JG_J^p=\sum_J\int G_J^p.
\end{equation}
All applications of Minkowski have exponent at least one.
For $q=p$, Tonelli gives the exact equality of the two defining
norms. Finally, if $q_0\leq q_1$ then
\[
 \|\{G_J(z)\}\|_{\ell^{q_1}}\leq\|\{G_J(z)\}\|_{\ell^{q_0}},\qquad
 \|\{\|G_J\|_p\}\|_{\ell^{q_1}}
 \leq\|\{\|G_J\|_p\}\|_{\ell^{q_0}},
\]
which proves the two monotonicity assertions.
\end{proof}

\begin{corollary}\label{cor:Hardy-Besov}
For $0<p<\infty$,
\begin{equation}\label{eq:Hardy-Besov}
 \dot B^{0,\min(p,2)}_{p,F}
       \hookrightarrow H^p_F
       \hookrightarrow\dot B^{0,\max(p,2)}_{p,F},
\end{equation}
with $H^p_F=L^p$ for $1<p<\infty$. If
$0<q_0\leq2\leq q_1\leq\infty$, then
$\dot F^{0,q_0}_{p,F}\hookrightarrow H^p_F
 \hookrightarrow\dot F^{0,q_1}_{p,F}$.
\end{corollary}
\begin{proof}
The flag Littlewood--Paley theorem \cite{HL08,HL10} gives
$\dot F^{0,2}_{p,F}=H^p_F$ and $H^p_F=L^p$ for $p>1$; the
agreement of the sampled and continuous square functions
follows from Lemma~\ref{lem:grid}, the Peetre estimate and the
vector-valued maximal inequality \cite{FS71}. Applying
Proposition~\ref{prop:FB-embeddings} with $s=0,q=2$ gives the
stated embeddings.
\end{proof}

\end{document}